%% file: main.tex
\documentclass[journal]{IEEEtran}
\input{packages}
\usepackage{etoolbox}

\newtoggle{TR}
\toggletrue{TR}
\title{\LARGE \bf Convex Optimization for Parameter Synthesis in MDPs}
\author{Murat Cubuktepe, Nils Jansen, Sebastian Junges, Joost-Pieter Katoen, Ufuk Topcu\thanks{M. Cubuktepe and U. Topcu are with the Department of Aerospace Engineering
		and Engineering Mechanics, Austin, USA. N. Jansen is with the Department of Software Science, Radboud University Nijmegen, Nijmegen, the Netherlands. S. Junges is with the Department of Electrical Engineering and Computer Sciences, University of California at Berkeley, Berkeley, USA. J.-P. Katoen is with the Departement of Computer Science, RWTH Aachen University, Aachen, Germany. email:(\{mcubuktepe,utopcu\}@utexas.edu, n.jansen@science.ru.nl, sjunges@berkeley.edu, katoen@cs.rwth-aachen.de).}
}

\input{macros}
\begin{document}

\maketitle\begin{flushright}
	
\end{flushright}
\thispagestyle{plain}
\pagestyle{plain}
\input{abstract}

\input{introduction}
\input{preliminaries}

\input{problem}

\input{nonlinear}

\input{convex_concave}

\input{affine_trust_region}
\input{convergence}

\input{experiments}
\input{conclusion}

\bibliographystyle{IEEEtran}
\bibliography{literature}
\iftoggle{TR}{
\input{appendix}

\vspace{-1.1cm}
}

\begin{IEEEbiography}[{\includegraphics[width=1.0in,height=1.25in,clip,keepaspectratio]{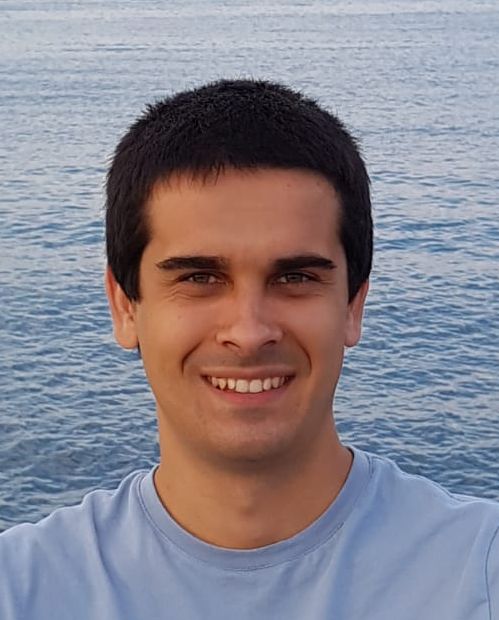}}]{Murat Cubuktepe} joined the Department of Aerospace Engineering at the University of Texas at Austin as a Ph.D. student in Fall 2015. He received his B.S degree in Mechanical Engineering from Bogazici University in 2015. His research is on the theoretical and algorithmic aspects of the design and verification of autonomous systems in the intersection of formal methods, convex optimization, and artificial intelligence.
\end{IEEEbiography}
\vspace{-1.1cm}
\begin{IEEEbiography}[{\includegraphics[width=1.0in,height=1.25in,clip,keepaspectratio]{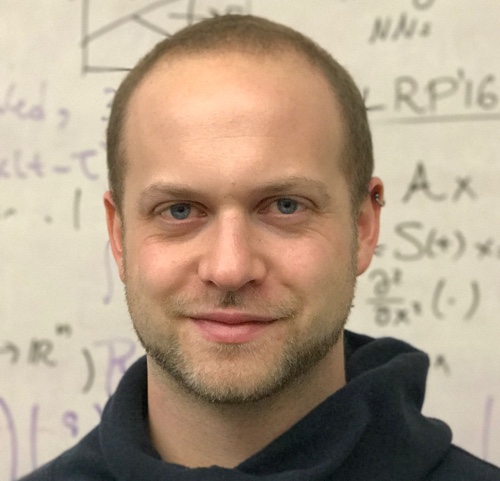}}]{Nils Jansen} is an assistant professor with the Institute for Computing and Information Science (iCIS) at the Radboud University, Nijmegen, The Netherlands. He received his Ph.D. in computer science with distinction from RWTH Aachen University, Germany, in 2015. Prior to Radboud University, he was a postdoctoral researcher and research associate with the Institute for Computational Engineering and Sciences at the University of Texas at Austin. His current research focuses on formal reasoning about safety aspects in machine learning and robotics. At the heart is the development of concepts inspired from formal methods to reason about uncertainty and partial observability.
\end{IEEEbiography}
\vspace{-1.1cm}
\begin{IEEEbiography}[{\includegraphics[width=1.0in,height=1.25in,clip,keepaspectratio]{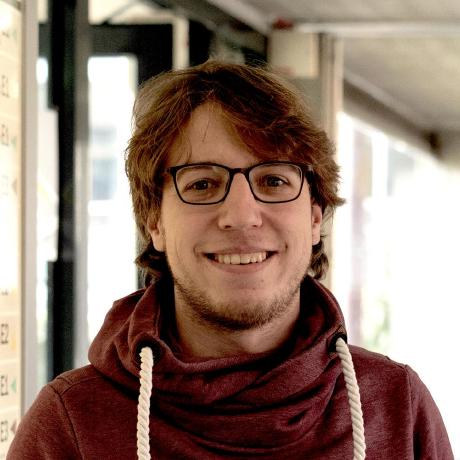}}]{Sebastian Junges} is a postdoctoral researcher at the University of Berkeley, California. In 2020, he received his PhD degree with distinction from RWTH Aachen University, Germany. His research focuses on the model-based analysis of controllers for uncertain environments, either applied for runtime assurance or as part of the design process. In his research, he applies and extends ideas from formal methods, in particular from satisfiability checking and probabilistic model checking.
\end{IEEEbiography}
\vspace{-1.2cm}
\begin{IEEEbiography}[{\includegraphics[width=1.0in,height=1.25in,clip,keepaspectratio]{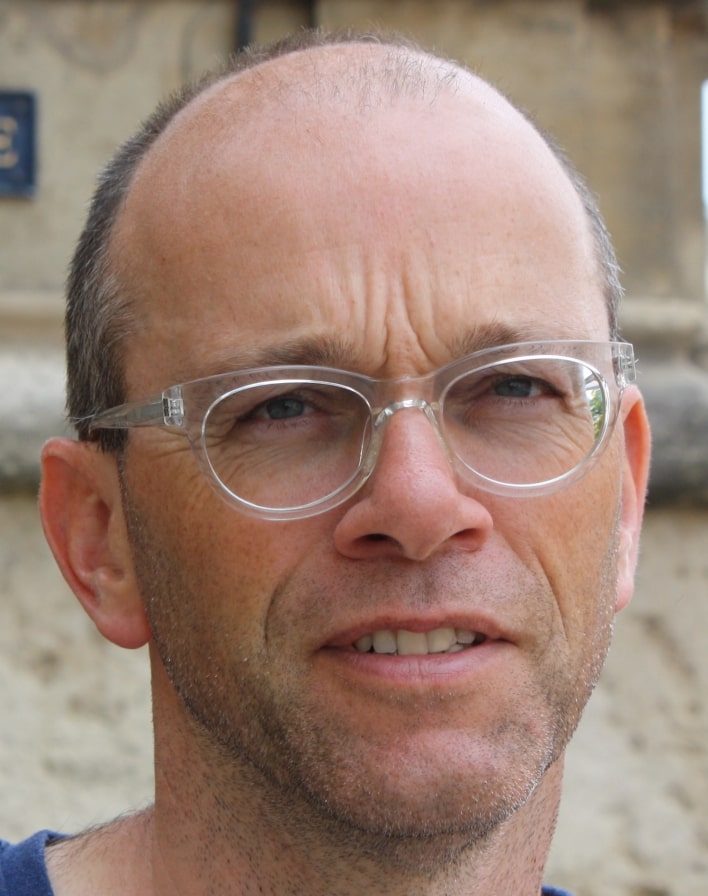}}]{Joost-Pieter Katoen} is a Distinguished Professor with RWTH Aachen University, Germany, and holds a parttime professorship at the University of Twente, The Netherlands. He received a honorary doctorate degree from Aalborg University, Denmark. His research interests include formal methods, model checking, concurrency theory, and probabilistic computation. He coauthored more than 180 conference papers, 75 journal papers, and the book ``Principles of Model Checking.'' Prof. Katoen is the Chairman of the steering committee of ETAPS, and steering committee member of the conferences CONCUR, QEST, and FORMATS. He is a member of Academia Europaea and holds an ERC Advanced Research Grant (2017).
\end{IEEEbiography}
\vspace{-1.2cm}
\begin{IEEEbiography}[{\includegraphics[width=1.0in,height=1.25in,clip,keepaspectratio]{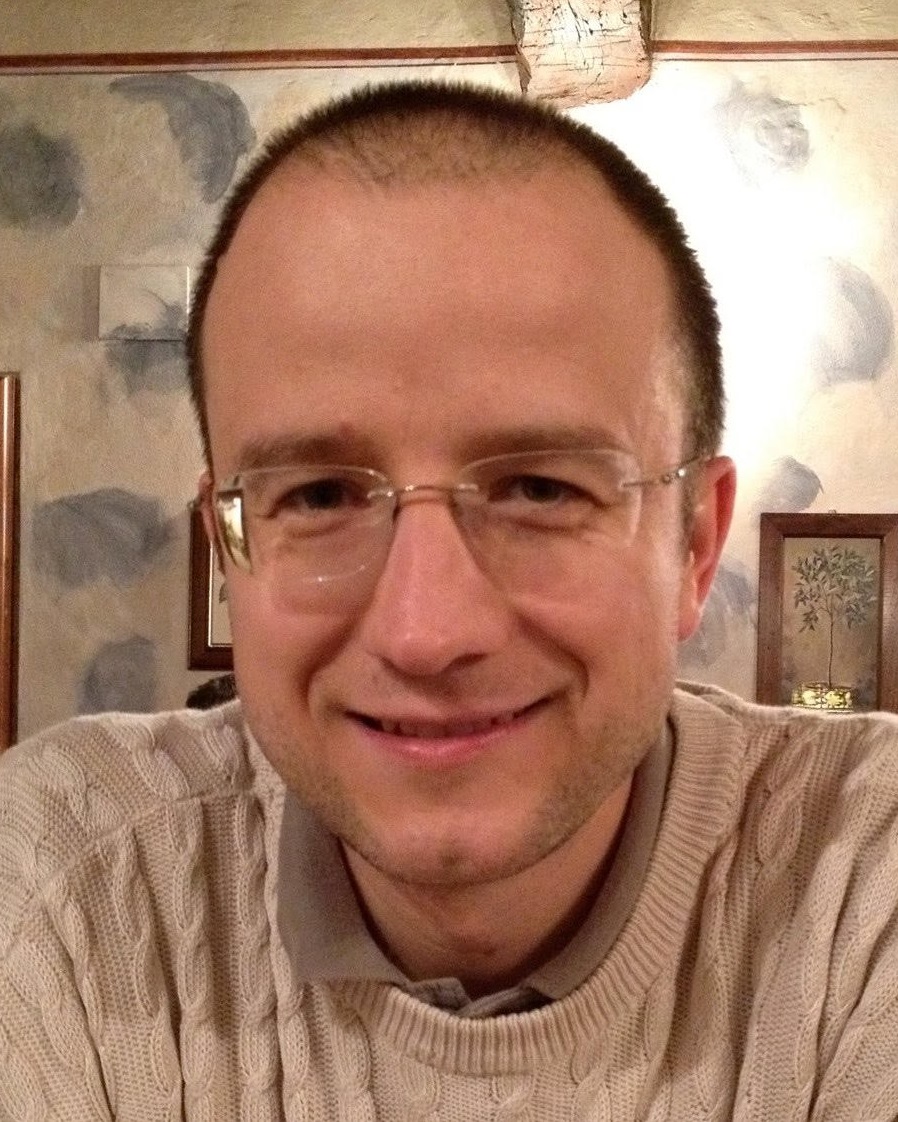}}]{Ufuk Topcu} joined the Department of Aerospace Engineering at the University of Texas at Austin as an assistant professor in Fall 2015. 
	He received his Ph.D. degree from the University of California at Berkeley in 2008. 
	He held research positions at the University of Pennsylvania and California Institute of Technology. 
	His research focuses on the theoretical, algorithmic and computational aspects of design and verification of autonomous systems through novel
	connections between formal methods, learning theory and controls.
\end{IEEEbiography}

\end{document}

%% file: packages.tex
\usepackage{amsmath,amssymb,amsfonts}
\usepackage{mathtools}
\usepackage{amsthm}
\usepackage[pdfpagelabels=true,linktocpage]{hyperref}
\usepackage[usenames,dvipsnames]{color}
\usepackage{xcolor}
\usepackage{xspace}
\usepackage{microtype}
\usepackage{todonotes}
\usepackage{colonequals}
 \usepackage{booktabs}
\usepackage{cite}

\usepackage{multirow}
\usepackage{multicol}

\usepackage{paralist}
\usepackage{diagbox}

\usepackage{url}
\usepackage{graphicx}
\usepackage{algorithm}
\usepackage{algpseudocode}

\usepackage{listings}
\usepackage{nicefrac}
\usepackage{tikz} 
\usepackage{pgfplots}
\usepackage{mdframed}
\pgfplotsset{compat=1.15}
\usepackage{tikzscale}
\usepackage{subcaption}

\usetikzlibrary{arrows,decorations.pathmorphing,positioning,fit,trees,shapes,shadows,automata,calc,arrows.meta} 

\usetikzlibrary{decorations.markings}
\tikzset{outline/.style args={#1}{%
  draw=#1,thick,fill=#1!50},initial text={}}
\input{die}

\tikzset{>=Stealth}

%% file: die.tex
\tikzset{
  dot hidden/.style={},
  line hidden/.style={},
  dice hidden/.style={},
  dot color/.style={dot hidden/.append style={color=#1}},
  dot color/.default=black,
  line color/.style={line hidden/.append style={color=#1}},
  line color/.default=black,
  dice color/.style={dice hidden/.append style={color=#1,fill}},
  dice color/.default=white
}\def\dotsize{0.1}
\newcommand{\drawdie}[2][]{%
\begin{tikzpicture}[x=1em,y=1em,#1]
  \draw 	[thick, rounded corners=0.5,line hidden,dice hidden] (0,0) rectangle (1,1);
  \ifodd#2
    \fill[dot hidden] (0.5,0.5) circle (\dotsize);
  \fi
  \ifnum#2>1
  \fill[dot hidden] (0.25,0.25) circle (\dotsize);
  \fill[dot hidden] (0.75,0.75) circle (\dotsize);
  \ifnum#2>3
    \fill[dot hidden] (0.25,0.75) circle (\dotsize);
    \fill[dot hidden] (0.75,0.25) circle (\dotsize);
    \ifnum#2>5
      \fill[dot hidden] (0.75,0.5) circle (\dotsize);
      \fill[dot hidden] (0.25,0.5) circle (\dotsize);
    \fi
  \fi
\fi
\end{tikzpicture}
}

%% file: macros.tex
\newcommand{\ie}{i.e.\@\xspace}

\newcommand{\prophesy}{\textrm{PROPhESY}\xspace}
\newcommand{\prism}{\textrm{PRISM}\xspace}
\newcommand{\storm}{\textrm{Storm}\xspace}
\newcommand{\param}{\textrm{PARAM}\xspace}
\newcommand{\tool}[1]{\textrm{#1}\xspace}
 \newtheorem{definition}{Definition}
\newtheorem{problem}{Problem}
\newtheorem{theorem}{Theorem}
\newtheorem{remark}{Remark}

\newtheorem{proposition}{Proposition}
\newtheorem{example}{Example}
\newcommand{\TO}{TO}
\newcommand{\MO}{MO}
\newcommand{\highlight}[1]{\textcolor{blue!50!black}{\textbf{#1}}}

\newcommand{\epsgraph}{\varepsilon_\text{graph}}

\newcommand{\dtmc}{\mathcal{D}}

\newcommand{\p}{\ensuremath{\mathbb{P}}}
\newcommand{\pr}{\ensuremath{\mathrm{Pr}}}

\newcommand{\reachProp}[2]{\ensuremath{\p_{\leq #1}(\finally #2)}}
\newcommand{\reachProplT}{\ensuremath{\reachProp{\lambda}{T}}}

\newcommand{\reachPropSymbol}{\varphi_r}
\newcommand{\ereachPropSymbol}{\varphi_c}

\newcommand{\expRewProp}[2]{\ensuremath{\EV_{\leq #1}(\finally #2)}}
\newcommand{\expRewPropkT}{\ensuremath{\expRewProp{\kappa}{T}}}
\newcommand{\rewFunction}{\ensuremath{{c}}}

\newcommand{\finally}{\lozenge}

\makeatletter
\renewcommand\fs@ruled{\def\@fs@cfont{\bfseries}\let\@fs@capt\floatc@ruled
	\def\@fs@pre{{\color{black}\hrule height.8pt depth0pt \kern2pt}}%
	\def\@fs@post{{\color{black}\kern2pt\hrule\relax}}%
	\def\@fs@mid{{\kern2pt\color{black}\hrule\kern2pt}}%
	\let\@fs@iftopcapt\iftrue}
\makeatother

\newcommand{\R}{\mathbb{R}}

\newcommand{\Ireal}{[0,\, 1]\subseteq\mathbb{R}}  
\newcommand{\EV}{\ensuremath{\mathbb{E}}}

\newcommand{\Distr}{\mathit{Distr}}

\newcommand{\distDom}{X}

\newcommand{\distFunc}{\mu}
\newcommand{\distDomElem}{x}

\newcommand{\Paramvar}{\ensuremath{{V}}\xspace}        

\newcommand{\sinit}{s_{\mathit{I}}} 
\newcommand{\mdp}{\mathcal{M}}

\newcommand{\pMdpInit}[1][]{\ensuremath{\mdp{#1}=(S{#1},\,\sinit{#1},\Act,\Paramvar,\probmdp{#1})}}
\newcommand{\probmdp}{\mathcal{P}}

\newcommand{\sched}{\ensuremath{\sigma}}
\newcommand{\Sched}{\ensuremath{\mathit{Str}}}

\newcommand{\Act}{\ensuremath{\mathit{Act}}}
\newcommand{\ActS}{\ensuremath{\mathit{A}}}

\newcommand{\act}{\ensuremath{\alpha}}
\newcommand{\pmdp}{\ensuremath{\mathcal{P}}}

\newcommand{\added}[1]{{#1}}

\DeclareMathAlphabet{\mathpzc}{OT1}{pzc}{m}{it}
\def\presuper#1#2%
  {\mathop{}%
   \mathopen{\vphantom{#2}}^{#1}%
   \kern-\scriptspace%
   #2}

\newcommand{\Statey}{\Statex\hspace*{\ALG@thistlm}}

\IEEEoverridecommandlockouts

%% file: abstract.tex
\begin{abstract}
Probabilistic model checking aims to prove whether a Markov decision process (MDP) satisfies a temporal logic specification.
The underlying methods rely on an often unrealistic assumption that the MDP is precisely known.
Consequently, parametric MDPs (pMDPs) extend MDPs with transition probabilities that are functions over unspecified parameters.
The parameter synthesis problem is to compute an instantiation of these unspecified parameters such that the resulting MDP satisfies the temporal logic specification.
We formulate the parameter synthesis problem as a quadratically constrained quadratic program (QCQP), which is nonconvex and is NP-hard to solve in general.
We develop two approaches that iteratively obtain locally optimal solutions.
The first approach exploits the so-called convex-concave procedure (CCP), and the second approach utilizes a sequential convex programming (SCP) method.
The techniques improve the runtime and scalability by multiple orders of magnitude compared to black-box CCP and SCP by merging ideas from convex optimization and probabilistic model checking.
We demonstrate the approaches on a satellite collision avoidance problem with hundreds of thousands of states and tens of thousands of parameters and their scalability on a wide range of commonly used benchmarks.
\end{abstract}

%% file: introduction.tex
\section{Introduction}
\label{sec:introduction}

Markov decision processes (MDPs) are widely studied models for sequential decision-making~\cite{Put94}.
MDPs have been used, for example, in robotic applications~\cite{ding2014optimal}, aircraft collision avoidance systems~\cite{von2014analyzing}, and Mars robot
missions~\cite{nilsson2018toward}.
The formal verification of temporal logic specifications has been extensively studied for MDPs~\cite{BK08}. 
Such specifications are able to express properties such as ``the maximum probability of reaching a set of goal states without colliding with an obstacle is more than 99\%'' or ``the minimum expected time of reaching the target location is less than 10 seconds''.
A crucial, yet possibly unrealistic, assumption in probabilistic model checking is that the transition and cost functions of the MDP are exactly known.
However, these values are often estimated from data and may not be obtained exactly.

More general models express cost and probabilities as functions over parameters whose values are left unspecified~\cite{Daw04, lanotte, param_sttt}. 
Such \emph{parametric MDPs} (pMDPs) describe uncountable sets of MDPs. 
A well-defined instantiation of the parameters yields an \emph{instantiated, parameter-free} MDP.
Applications of pMDPs include adaptive software systems~\cite{calinescu-et-al-cacm-2012}, sensitivity analysis~\cite{su-et-al-icse-2016-qosevaluation}, optimizing randomized distributed algorithms~\cite{DBLP:conf/srds/AflakiVBKS17}, and synthesis of finite-memory strategies for partially observable MDPs (POMDPs)~\cite{DBLP:journals/corr/abs-1710-10294}.\looseness=-1

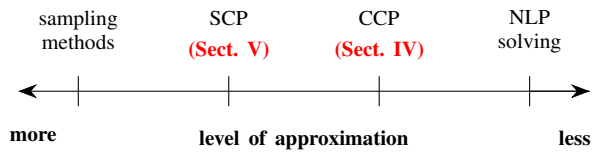
\begin{figure}
\centering
\begin{tikzpicture}[>={Stealth[scale=1.5]}]
	\draw (0,0) -- (6.8,0);
	\draw [-{Stealth[scale=1.5]}] (0,0) -- (-0.8,0);
	\draw [-{Stealth[scale=1.5]}] (6.0,0) -- (6.8,0);
	\draw (0,-.2) -- (0,.2);
	
	\draw (2.0,-.2) -- (2.0,.2);
	\draw (4.0,-.2) -- (4.0,.2);
	\draw (6.0,-.2) -- (6.0,.2);
	\node[anchor=north,text width=3.4cm,align=center] at (3.0,-0.4) {\textbf{\footnotesize{level of approximation}}};	
	
	\node[anchor=north,text width=1.4cm,align=center] at (-0.6,-0.4) {\textbf{\footnotesize{more}}};	
		\node[anchor=north,text width=1.4cm,align=center] at (6.6,-0.4) {\textbf{\footnotesize{less}}};	
	\node[anchor=north,text width=1.3cm,align=center] at (0,1.2) {\baselineskip=10.0pt\footnotesize{sampling methods}\par};
	\node[anchor=north,text width=1.3cm,align=center] at (2.0,1.2) {\footnotesize{SCP\\\textbf{\textcolor{red}{(Sect.~\ref{sec:scp})}}}};
	\node[anchor=north,text width=1.3cm,align=center] at (4.0,1.2) {\footnotesize{CCP\\\textbf{\textcolor{red}{(Sect.~\ref{sec:ccp})}}}};
	\node[anchor=north,text width=1.3cm,align=center] at (6.0,1.2) {\baselineskip=10.0pt\footnotesize{NLP solving}\par};
\end{tikzpicture}
	\caption{The spectrum of the solution approaches for the parameter synthesis problem. }	
\label{fig:spectrum}
\end{figure}

For a given finite-state pMDP, the \emph{parameter synthesis problem} is to compute a parameter instantiation such that the instantiated MDP satisfies a given temporal logic specification.
Solving this problem is ETR-complete, i.e., as hard as finding a root to a multivariate polynomial~\cite{winkler2019complexity}. 
Consequently, the problem is NP-hard and in PSPACE.
A straightforward approach to the parameter synthesis problem relies on an exact encoding into a nonlinear programming problem (NLP)~\cite{bartocci2011model} or a satisfiability modulo-theories formula~\cite{winkler2019complexity}. 
These approaches are, in general, limited to a few states and parameters.

To solve the parameter synthesis problem, it suffices to \emph{guess} a correct parameter instantiation. 
This insight has led to an adaptation of sampling-based techniques to the parameter synthesis problem~\cite{hahn2011synthesis}, most prominently particle swarm optimization (PSO)~\cite{chen2013model}.
After guessing an instantiation, one can efficiently verify the resulting associated parameter-free MDP to determine whether the specification is satisfied. 
These techniques can handle millions of states, but are restricted to a few parameters. 
Their performance degrades significantly with increasing parameters, e.g., with more than ten~\cite{hahn2011synthesis,chen2013model}.

\emph{How do these methods construct a good guess?}
One natural method is to \emph{solve} the original problem exactly. 
More precisely, exact methods translate the parameter synthesis problem into an equivalent nonlinear program. 
The advantage of these methods is that they only need one iteration. 
However, this iteration is, in general, very costly.
On the other hand, sampling-based methods completely disregard the model structure. 
Technically, they sample parameter instantiations from some prior, e.g., a uniform distribution from all samples. 
The advantage is that every iteration is very fast, but one may need a tremendous amount of iterations or samples.
These two approaches for the parameter synthesis problem are the extremal instances on the level of approximation, see Fig.~\ref{fig:spectrum}.


We develop methods that provide a trade-off between these two approaches by exploiting ideas from convex optimization.
More concretely, we approximate the NLP as a convex optimization problem and use the solution of this problem to find candidates of parameter instantiations. 
These instantiations are then verified using techniques from \emph{probabilistic model checking}~\cite{BK08}.
The resulting methods utilize the model structure of the pMDP, and fewer iterations are necessary compared to sampling-based methods, even though each iteration may take longer than the sampling-based methods.\looseness=-1\vspace{-0.1cm}
\subsection{Contributions}
We provide two solutions to the parameter synthesis problem.
In our first approach, we transform the NLP into a \emph{quadratically-constrained quadratic program} (QCQP). 
However, the resulting QCQP is nonconvex and is NP-hard to solve~\cite{alizadeh2003second,lobo1998applications}.
To obtain a locally optimal solution to this QCQP, we use the so-called convex-concave procedure (CCP)~\cite{lipp2016variations}.
To that end, we reformulate the nonconvex QCQP as a \emph{difference-of-convex} (DC) problem.
All constraints and the objective of a DC problem are a difference of two convex functions.
CCP computes a locally optimal solution to the resulting DC problem by convexifying it as a convex quadratic problem.
The resulting convex quadratic problem can be solved by state-of-the-art solvers such as \tool{Gurobi}~\cite{gurobi}.
We also integrate the CCP procedure with probabilistic model checking, which yields a speedup of multiple orders of magnitude compared to the existing CCP solvers.
This approach was published as a preliminary conference paper~\cite{cubuktepe2018synthesis}, and our presentation is partially based on \cite{Jun20}.


In our second approach, we exploit a sequential convex programming (SCP) method~\cite{yuan2015recent,mao2018successive,chen2013optimality} to solve the parameter synthesis problem.
We convexify the nonconvex QCQP into a \emph{linear program} (LP) by linearizing the nonconvex constraints around a previous solution.
Similar to CCP, SCP iteratively computes a locally optimal solution to the nonconvex QCQP.
However, unlike CCP, the linearization in SCP does not over approximate the functions in the constraints
Therefore, a feasible solution to the linearized problem in SCP may be infeasible to the parameter synthesis problem, unlike in CCP.
Existing SCP methods can ensure the correctness of the solution only when the starting point is feasible, which amounts to solving the parameter synthesis problem.
In this paper, we address the key critical shortcomings of the existing SCP methods.
Specifically, the solution obtained from SCP may not be feasible to the parameter synthesis problem due to approximation errors in linearization and potential numerical issues while solving the linearized problem.
First, we use so-called trust region constraints~\cite{yuan2015recent,mao2018successive,chen2013optimality} to ensure that the linearized problem accurately approximates the nonconvex QCQP.
Second, we integrate a probabilistic model checking step into SCP, similar to our first approach.
Instead, we use the values for each parameter after solving the linearized problem and model check the instantiated MDP.
We check whether the instantiation improves the probability of satisfying the specification compared to the previous instantiation.
We use these values as an input for the next iteration if the probability is improved.
Otherwise, we contract the radius of the trust region constraints and re-solve the linearized problem. 
We discuss the convergence properties of the proposed CCP and SCP method.\looseness=-1

We integrate the proposed CCP and SCP methods with a probabilistic model checker in the tool \prophesy~\cite{dehnert-et-al-cav-2015}. 
In particular, an extensive empirical evaluation on a broad range of benchmarks shows that the CCP and SCP method can solve the parameter synthesis problem for models with hundreds of thousands of states and tens of thousands of parameters as opposed to few parameters for the existing parameter synthesis tools. 
Thus, the resulting methods (1) solve multiple orders of magnitude larger problems compared to other parameter synthesis tools, (2) yield an improvement of multiple orders of magnitude in runtime compared to just using CCP and SCP as a black box, and (3) ensure the correctness of the solution and have favourable convergence properties.
%
\vspace{-0.15cm}
\subsection{Related work}
Traditionally, approaches for solving the parameter synthesis problems have been built around the notion of abstracting the parametric model into a solution function, similar to our approaches.
The solution function is the probability of satisfying the temporal logic specification as a function of the model parameters~\cite{Daw04,param_sttt,dehnert-et-al-cav-2015,DBLP:journals/tse/FilieriTG16}. 
The solution function can be exploited the probabilistic model checking tools \tool{PARAM}~\cite{param_sttt}, \tool{PRISM}~\cite{KNP11} and \tool{Storm}~\cite{DBLP:conf/cav/DehnertJK017} to solve the parameter synthesis problem. 
This function is exponentially large in the parameters, and solving the problem is again exponential in the number of parameters, making the whole approach doubly exponential~\cite{DBLP:journals/iandc/BaierHHJKK20}. 
Consequently, these approaches typically can handle millions of states but only a \emph{handful of parameters}. 
Moreover, these approaches require a fixed policy or has to introduce a parameter for every state/action-pair in the MDP.

Orthogonally, 
Quatmann et al.~\cite{quatmann2016parameter} address an alternative parameter synthesis problem which focuses on proving the absence of parameter instantiations. 
The method iteratively solves  simple stochastic games. 
Spel et al.~\cite{DBLP:conf/atva/SpelJK19} consider proving that the parameters behave monotonically, allowing for faster sampling-based approaches. 
However, this method is limited to a few parameters.
A recent survey on parameter synthesis in Markov models can be found in~\cite{junges2019parameter}.

Further variations of parameter synthesis, e.g., consider statistical guarantees for parameter synthesis, often with some prior on the parameter values~\cite{bortolussi2018bayesian,calinescu2016fact,cubuktepe2020scenario}.
These approaches cannot provide the absolute guarantees on an answer that the methods in this paper provide.

Parametric MDPs generalize interval models~\cite{DBLP:conf/tacas/SenVA06,DBLP:journals/ipl/ChenHK13}. 
Such interval models have also been considered with convex uncertainties~\cite{seshia_et_al_cav_13,hahn2017multi,wu2008reachability,chen2013complexity}.
However, the resulting problems with interval models are easier to solve due to the lack of dependencies (or couplings) between parameters in different states.

Finally, similar convex-optimization-based methods to our approaches have been used to synthesize finite-memory strategies for POMDPs~\cite{amato2006solving,amato2010optimizing}.
However, our techniques can solve POMDPs that are significantly larger than these methods.

\subsection{Organization}
We first provide the necessary preliminaries in Section~\ref{sec:preliminaries}. 
Then we provide the problem statement and formulate the problem as a nonconvex QCQP in Section~\ref{sec:problem}.
We develop the CCP method in Section~\ref{sec:ccp} and the SCP method in Section~\ref{sec:scp}.
Section~\ref{sec:convergence} discusses the convergence rate of the methods. 
Section~\ref{sec:experiments} shows numerical examples and demonstrates the scalability of the approaches.
We conclude our paper and discuss possible future directions in Section~\ref{sec:conclusion}.

%% file: preliminaries.tex
\section{Preliminaries}
\label{sec:preliminaries}
A \emph{probability distribution} over a finite or countably infinite set $\distDom$ is a function $\distFunc\colon\distDom\rightarrow\Ireal$ with $\sum_{\distDomElem\in\distDom}\distFunc(\distDomElem)=1$. 
The set of all distributions on $\distDom$ is denoted by $\Distr(\distDom)$.
Let $V=\{v_1,\ldots,v_n\}$ be a finite set of \emph{variables} over the real numbers $\R$. The set of multivariate polynomials in $[0, 1]$ over $V$ is $\mathbb{Q}[V]$. An \emph{instantiation} for $V$ is a function $\mathbf{v}\colon V \rightarrow \R$.\looseness=-1

%
\begin{definition}[(Affine) pMDP]\label{def:pmdp}
A \emph{parametric Markov decision process (pMDP)} is a tuple $\pMdpInit$ with a finite set $S$ of \emph{states}, an \emph{initial state} $\sinit \in S$, a finite set $\Act$ of \emph{actions}, a finite set $\Paramvar$ of real-valued variables \emph{(parameters)} and a \emph{transition function} $\probmdp \colon S \times \Act \times S \rightarrow \mathbb{Q}[V]$.
A pMDP is \emph{affine} if $\probmdp(s,\act,s')$ is affine in $V$ for every $s,s'\in S$ and $\act\in \Act$.\looseness=-1
\end{definition}
%
%
%
For $s \in S$,  $\ActS(s) = \{\act \in \Act \mid \exists s'\in S.\,\probmdp(s,\,\act,\,s') \neq 0\}$ is the set of \emph{enabled} actions at $s$.
Without loss of generality, we require $\ActS(s) \neq \emptyset$ for $s\in S$.
If $|\ActS(s)| = 1$ for all $s \in S$, $\mdp$ is a \emph{parametric discrete-time Markov chain (pMC)}. 
MDPs can be equipped with a state--action \emph{cost function} $\rewFunction \colon S \times \Act \rightarrow \R_{\geq 0}$.

A pMDP $\mdp$ is a \emph{Markov decision process (MDP)} if the transition function yields \emph{well-defined} probability distributions, \ie, $\probmdp \colon S \times \Act \times S \rightarrow [0,1]$ and $\sum_{s'\in S}\probmdp(s,\act,s') = 1$ for all $s \in S$ and $\act \in \ActS(s)$. 
Applying an \emph{instantiation} $\mathbf{v}\colon V \rightarrow \R$ to a pMDP $\mdp$ yields an \emph{instantiated MDP} $\mdp[\mathbf{v}]$ by replacing each $f\in\mathbb{Q}[V]$ in $\mdp$ by $f[\mathbf{v}]$.
An instantiation $\mathbf{v}$ is \emph{well-defined} for $\mdp$ if 
the resulting model $\mdp[\mathbf{v}]$ is an MDP.

To define measures on
MDPs, nondeterministic choices are resolved by a so-called \emph{strategy} $\sched\colon S\rightarrow\Act$ with $\sched(s) \in \ActS(s)$.
The set of all strategies over $\mdp$ is $\Sched^\mdp$.
%
%
%
For the measures in this paper, memoryless deterministic strategies suffice~\cite{BK08}.
Applying a strategy to an MDP yields an \emph{induced Markov chain} where all nondeterminism is resolved.
%

For an MC $\dtmc$, the \emph{reachability specification} $\reachPropSymbol=\reachProplT$ asserts that a set $T \subseteq S$ of \emph{target states}  is reached with probability at most $\lambda\in [0,1]$.
If $\reachPropSymbol$ holds for $\dtmc$, we write $\dtmc\models\reachPropSymbol$.
Accordingly, for an \emph{expected cost specification}, $\ereachPropSymbol=\expRewProp{\kappa}{G}$, $\dtmc\models\ereachPropSymbol$ holds if and only if the expected cost of reaching a set $G \subseteq S$ is bounded by $\kappa \in \R$.
We use standard measures and definitions as in~\cite[Ch.\ 10]{BK08}.
We note that linear temporal logic specifications can be
reduced to reachability specifications, and we refer the reader to~\cite{BK08} for a detailed introduction.
%
%
%
%
An MDP $\mdp$ satisfies a specification $\varphi$, written $\mdp\models\varphi$, if and only if \emph{for all} strategies $\sched\in\Sched^\mdp$ it holds that $\mdp^\sched\models\varphi$. 

%% file: problem.tex
\section{Formal Problem Statement}
\label{sec:problem}
In this section, we state the parameter synthesis problem, which is to compute a parameter instantiation such that the instantiated MDP satisfies the given temporal logic specification.
We then discuss the nonlinear program formulation of the parameter synthesis problem, which forms the basis of our solution methods.
\begin{problem}[Parameter synthesis problem]\label{prob:pmdpsyn}
Given a pMDP $\pMdpInit$, and a reachability specification $\reachPropSymbol=\reachProplT$, 
 compute a well-defined
instantiation $\mathbf{v}\colon V \rightarrow \R$ for $\mdp$ such that $\mdp[\mathbf{v}]\models\reachPropSymbol$.
\end{problem}
Intuitively, we seek an instantiation of the parameters that satisfies $\reachPropSymbol$ for all possible strategies in the instantiated MDP.
We show necessary adaptions for an expected cost specification $\ereachPropSymbol=\expRewPropkT$ later.


For a given well-defined instantiation $\mathbf{v}$, Problem~\ref{prob:pmdpsyn} can be solved by verifying whether $\mdp[\mathbf{v}]\models\reachPropSymbol$. 
The standard formulation uses a linear program (LP) to minimize the probability $p_{\sinit}$ of reaching the target set $T$ from the initial state $\sinit$ while ensuring 
that this probability is realizable under any strategy~\cite[Ch.\ 10]{BK08}.
The straightforward extension of this approach to pMDPs to \emph{compute} a satisfiable instantiation $\mathbf{v}$ yields the following nonlinear program (NLP)~\cite{DBLP:conf/tacas/Cubuktepe0JKPPT17,cubuktepe2018synthesis} with the variables $p_s$ for $s \in S$, and the \emph{parameter variables} in $V$ in the transition function $\probmdp(s,\act,s')$ for $s,s' \in S$ and $\act \in \ActS$:
		\begin{align}
			\text{minimize} &\quad p_{\sinit}\label{eq:min_mdp}\\
			\text{subject to} &\nonumber \\
		p_s=1,\label{eq:targetprob_mdp}&\quad	\forall s\in T,	 & \\
	 \probmdp(s,\act,s')\geq 0,&\quad	\forall s,s'\in S\setminus T,\, \forall\act\in\ActS(s), &\label{eq:well-defined_probs_mdp}\\
	 \sum_{s'\in S}\probmdp(s,\act,s')=1,\label{eq:well-defined_probs_mdp1}	&\quad\forall s\in S\setminus T,\, 	\forall\act\in\ActS(s),	 \\
			\lambda \geq p_{\sinit},&  \label{eq:probthreshold_mdp}\\
	 p_s \geq \sum_{s'\in S}	\probmdp(s,\act,s')\cdot p_{s'}&\quad	\forall s\in S\setminus T,\,\forall \act\in\ActS(s).	\label{eq:probcomputation_mdp}
		\end{align}
For $s \in S$, the \emph{probability variable} $p_s\in[0,1]$ represents an upper bound of the probability of reaching target set $T\subseteq S$.
The \emph{parameters} in the set $V$ enter the NLP as part of the functions from $\mathbb{Q}[V]$ in the transition function $\probmdp$.
The constraint~\eqref{eq:probthreshold_mdp} ensures that the probability of reaching $T$ is below the threshold $\lambda$.
This constraint is optional for stating the problem, but we use the constraint in our methods for finding a parameter instantiation that satisfies the specification $\varphi$.
We minimize $p_{\sinit}$ to assign probability variables their minimal values with respect to the parameters $V$.

The probability of reaching a state in $T$ from $T$ is set to one~\eqref{eq:targetprob_mdp}.
The constraints~\eqref{eq:well-defined_probs_mdp} and~\eqref{eq:well-defined_probs_mdp1} ensure \emph{well-defined} transition probabilities.
Recall that $\probmdp(s,\act,s')$ is an affine function in $V$.
Therefore, the constraints~\eqref{eq:well-defined_probs_mdp} and~\eqref{eq:well-defined_probs_mdp1} only depend on the parameters in $V$, and they are affine in the parameters.
Constraint~\eqref{eq:probthreshold_mdp} is optional but necessary later, and ensures that the probability of reaching $T$ is below the threshold $\lambda$.
For each state $s\in S\setminus T$ and action $\act\in\ActS(s)$, the probability induced by the \emph{maximizing scheduler} is a lower bound to the probability variables $p_s$~\eqref{eq:probcomputation_mdp}.
To assign probability variables to their minimal values with respect to the parameters in $V$, $p_{\sinit}$ is minimized in the objective~\eqref{eq:min_mdp}.
We state the correctness of the NLP in Proposition~1.
%
\begin{proposition}
The NLP in \eqref{eq:min_mdp} -- \eqref{eq:probcomputation_mdp} computes the \emph{minimal probability} of reaching $T$ under a \emph{maximizing} strategy, and an instantiation $\mathbf{v}$ is feasible to the NLP if and only if $\mdp[\mathbf{v}]\models\reachPropSymbol$.
\end{proposition}
\begin{proof}
The NLP in \eqref{eq:min_mdp} -- \eqref{eq:probcomputation_mdp} is an extension of the LP in~\cite[Theorem 10.105]{BK08}. We refer to~\cite[Theorem 4.20]{Jun20} for a formal proof.
\end{proof}

%


%
%
%
\begin{remark}[Graph-preserving instantiations]
	In the LP formulation for MDPs, states with probability $0$ to reach $T$ are determined via a preprocessing on the underlying graph, and their probability variables are set to zero to ensure that the variables encode the actual reachability probabilities.
	We do the same. 
	This preprocessing requires the underlying graph of the pMDP to be preserved under any valuation of the parameters. 
	Thus, as in~\cite{param_sttt,dehnert-et-al-cav-2015}, we consider only graph-preserving valuations. 
	Concretely, we exclude valuations $\mathbf{v}$ with $f[\mathbf{v}]=0$ for $f\in\probmdp(s,\act,s')$ for all $s,s'\in S$ and $\act\in\Act$.
We replace the set of constraints \eqref{eq:well-defined_probs_mdp} by%
\begin{align}
 	\forall s,s'\in S.\, \forall\act\in\ActS(s).	 &\quad \probmdp(s,\act,s')\geq \epsgraph,\label{eq:well-defined_eps}
 \end{align}
 where $\epsgraph>0$ is a small constant. 
 
\end{remark}

\added{We demonstrate the constraints for the NLP in~\eqref{eq:min_mdp} -- \eqref{eq:well-defined_eps} for a pMC by Example~\ref{ex:nlp}.}

\begin{example}\label{ex:nlp}
Consider the pMC in Fig.~\ref{fig:pmc_reform} with parameter set $V=\{v\}$, initial state $s_0$, and target set $T = \{s_3\}$. Let $\lambda$ be an arbitrary constant.
The NLP in~\eqref{eq:nlp-ex1} -- \eqref{eq:nlp-ex2} minimizes the probability of reaching $s_3$ from the initial state:
\begin{align}
			\textnormal{minimize} & \quad p_{s_0} \label{eq:nlp-ex1} \\
			\textnormal{subject to} &\quad p_{s_3}=1,\\
			&\quad \lambda \geq p_{s_0} \geq v\cdot p_{s_1}, \\
			&\quad p_{s_1} \geq (1-v)\cdot p_{s_2},\\
			&\quad  p_{s_2} \geq v\cdot p_{s_3},  \\
			&\quad 1-\epsgraph \geq v\geq \epsgraph. \label{eq:nlp-ex2}
		\end{align}
\end{example}
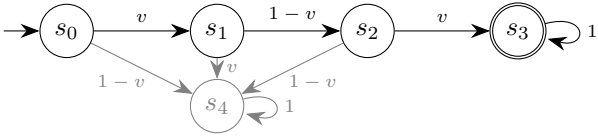
\begin{figure}[t]
	\centering
	\input{pics/nonmonotonic_dtmc1}	
	\caption{A pMC with a single parameter $v$.}
	\label{fig:pmc_reform}
\end{figure}
\paragraph{Expected cost specifications}
The NLP in \eqref{eq:min_mdp} -- \eqref{eq:well-defined_eps} considers reachability probabilities.
If we have instead an expected cost specification $\ereachPropSymbol=\expRewProp{\kappa}{G}$, we replace \eqref{eq:targetprob_mdp}, \eqref{eq:probthreshold_mdp}, and \eqref{eq:probcomputation_mdp} in the NLP by the following constraints:
\begin{align}
	p_s=0,&\;\forall s\in G.	  \label{eq:targetrew}\\
p_s\geq  c(s,\act) + \hspace{-0.08cm} \sum_{s'\in S}\hspace{-0.04cm}	\probmdp (s,\act,s')\cdot p_{s'},&\;	\forall s\in S\setminus G,\, \forall\act\in\ActS(s),	
\label{eq:rewcomputation}\raisetag{10pt}\\
\kappa \geq p_{\sinit}&\label{eq:strategyah:lambda}.
\end{align}
We have $p_s\in\R$, as these variables represent the expected cost to reach $G$. 
At $G$, the expected cost is set to zero \eqref{eq:targetrew}, and the actual expected cost for other states is a lower bound to $p_s$ \eqref{eq:rewcomputation}.
Finally, 
$p_{\sinit}$ is bounded by the threshold $\kappa$.

%% file: pics/nonmonotonic_dtmc1.tex
\begin{tikzpicture}[scale=1, nodestyle/.style={draw,circle},baseline=(s0),>={Stealth[scale=1.5]}]

    \node [nodestyle,initial,-{Stealth[scale=1.5]}] (s0) at (0,0) {$s_0$};
    \node [nodestyle] (s1) [on grid, right=2cm of s0] {$s_1$};
    \node [nodestyle] (s2) [on grid, right=2cm of s1] {$s_2$};
    \node [nodestyle ,accepting] (s3) [on grid, right=2cm of s2] {$s_3$};
    \node [nodestyle, gray] (s4) [on grid, below=1cm of s1] {$s_4$};

    \draw[-{Stealth[scale=1.5]}] (s0) -- node [auto] {\scriptsize$v$} (s1);
    \draw[-{Stealth[scale=1.5]},gray] (s0) -- node [below,pos=0.3, yshift=-0.1cm] {\scriptsize$1-v$} (s4);

    \draw[-{Stealth[scale=1.5]}] (s1) -- node [auto] {\scriptsize$1-v$} (s2);
    \draw[-{Stealth[scale=1.5]},gray] (s1) -- node [auto] {\scriptsize$v$} (s4);
    
    \draw[-{Stealth[scale=1.5]}] (s2) -- node [auto] {\scriptsize$v$} (s3);
    \draw[-{Stealth[scale=1.5]},gray] (s2) -- node [below,pos=0.3, yshift=-0.1cm] {\scriptsize$1-v$} (s4);

    \draw(s3) edge[loop right,-{Stealth[scale=1.5]}] node [right] {\scriptsize$1$} (s3);
    \draw(s4) [gray,-{Stealth[scale=1.5]}] edge[loop right,-{Stealth[scale=1.5]}] node [right] {\scriptsize$1$} (s3);
    
\end{tikzpicture}

%% file: nonlinear.tex
For the remainder of this paper, we restrict pMDPs to be affine, see Definition~\ref{def:pmdp}.
For an affine pMDP $\mdp$, the functions in the resulting NLP ~\eqref{eq:min_mdp} -- \eqref{eq:probthreshold_mdp} for pMDP synthesis are affine in $V$.
However, the functions in the constraints~\eqref{eq:probcomputation_mdp} are \emph{quadratic}, as a result of multiplying affine functions occurring in $\probmdp$ with the probability variables $p_{s'}$.
Therefore, the problem in~\eqref{eq:min_mdp} -- \eqref{eq:probcomputation_mdp} is a quadratically constrained quadratic program (QCQP)~\cite{boyd_convex_optimization} and is generally nonconvex~\cite{cubuktepe2018synthesis}.

\begin{remark}
In the literature, pMDPs and pMCs appearing in benchmarks and case-studies are almost exclusively affine. 
Furthermore, from a complexity-theoretic point of view, solving for these pMDPs is as hard as when considering the general pMDP definition~\cite{winkler2019complexity}.  
We refer to~\cite[Sec. 5.1.1]{Jun20} for a discussion of subclasses of pMDPs.
\end{remark}

%% file: convex_concave.tex
	\section{Convex-Concave Procedure}
	\label{sec:ccp}
In this section, we present our solution based on the \emph{penalty convex-concave procedure} (CCP)~\cite{lipp2016variations}, which iteratively over-approximates a nonconvex optimization problem.
Specifically, we rewrite the quadratic functions in~\eqref{eq:probcomputation_mdp} as a \emph{sum of convex and concave functions} and linearize the concave functions.
The resulting convex problem can then be solved efficiently, and the process is iterated until a suitable solution is found.
However, the convergence conditions of CCP might be too conservative if the initial solution is infeasible.
Specifically, the obtained parameter instantiation and the instantiated MDP might solve the parameter synthesis problem, even though it cannot be certified by the convergence conditions of CCP.
Motivated by this fact, we integrate a model checking procedure into the CCP, which ensures the numerical stability of the solution, and certifies whether a computed instantiation solves the parameter synthesis problem.
We remark that the model checking procedure is also a critical part of our second approach, which is based on sequential convex programming.\looseness=-1





We depict the CCP approach for solving the nonconvex QCQP in~\eqref{eq:min_mdp} -- \eqref{eq:probcomputation_mdp} in Fig.~\ref{fig:ccploop}. 
The approach searches for a solution of the nonconvex QCQP by solving approximations of the QCQP in the form of convexified problems around some initial assignment for the parameters $\hat{\mathbf{v}}$ and probability variables $\hat{\mathbf{p}}$. 
After solving the convexified problem, we obtain the values of the penalty variables of this solution. \
If this penalty is zero, we have found a solution to the original nonconvex QCQP. 
Otherwise, we update (or guess) a new $\hat{\mathbf{v}}$ and  $\hat{\mathbf{p}}$ which we use to convexify the QCQP. 
We denote the obtained solution from the convexified problem for the parameter variables as $\mathbf{v}$ and for the probability variables as $\mathbf{p}$.
This loop may converge to a solution with positive values of the penalty variables, requiring restarting from another initial parameter instantiation. 
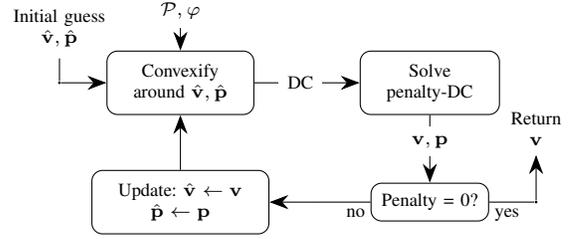
\begin{figure}
	\centering
	\input{figures/qcqpnaiveloop}
	\caption{Penalty CCP loop. It starts with some assignment for parameters $\hat{\mathbf{v}}$ and probability variables $\hat{\mathbf{p}}$ and iteratively solves convex problems until the penalty variables converge to zero.}
	\label{fig:ccploop}
\end{figure}

\subsection{Constructing a convex approximation}
We start with the construction of the penalty DC problem, then discuss updating the variables.
For compact notation, let $h(s,\act,s')$ be the quadratic function in $V$ and $p_{s'}$, i.e., 
\begin{align*}
h(s,\act,s')=\probmdp(s,\act,s') \cdot  p_{s'}
\end{align*}
for any $s, s' \in S, \act \in \ActS(s)$ whose right-hand is part of the constraint~\eqref{eq:probcomputation_mdp} in the nonconvex QCQP.
Note that $h(s,\act,s)$ is an affine function in $V$ for affine MDPs.
We first write this quadratic function as a difference of two convex functions.
For simplicity, let \begin{equation*}\probmdp(s,\act,s')=2d\cdot y+c, \text{ and } p_{s'}=z,\end{equation*}
where $y$ is the parameter variable, $z$ is the probability variable, and $c \in \R, d \in \R_{+}$ are constants.
We equivalently rewrite each
bilinear function $h(s,\act,s')$ as \begin{align*}
& 2d \cdot y z+d(y^2+z^2)-d(y^2+z^2)+c\cdot z\\
 =\;&d(y+z)^2-d(y^2+z^2)+c\cdot z.
\end{align*}
The function $ d(y+z)^2+c\cdot z$ is a \emph{quadratic convex function} in $y$ and $z$. 
In the remainder, let $h_{\textrm{cvx}}(s,\act,s') \colonequals d(y+z)^2$.
We show an example of the resulting DC problem in Example~\ref{ex:dc}.

\begin{example}
	\label{ex:dc}
	Recall the pMC in Fig.~\ref{fig:pmc_reform} and the QCQP from Example~\ref{ex:nlp}.  All quadratic constraints in the example are nonconvex. 
	We construct a DC problem with $d=0.5$ as
	\begin{align*}
	\textnormal{minimize} &\quad p_{s_0} \\
	\textnormal{subject to}&\quad p_{s_3}=1,\\
	&\quad \lambda \geq p_{s_0} \geq d(v+p_{s_1})^2-d(v^2+p^2_{s_1}),\\
	&\quad p_{s_1} \geq p_{s_2}+d(v-p_{s_2})^2-d(v^2+p^2_{s_2}),\\
	&\quad p_{s_2} \geq d(v+p_{s_3})^2-d(v^2+p^2_{s_3}),\\
	&\quad 1-\epsgraph  \geq v \geq \epsgraph.
	\end{align*}
\end{example}

The remaining term $-d(y^2+z^2)$, denoted by $h_{\textrm{ccv}}(s,\act,s')$,  is concave, and we have to convexify it to obtain a convex QCQP.
We compute an affine approximation in the form of a linearization of the term around an assignment $\langle \hat{y},\hat{z}\rangle$ by  \begin{equation}d(\hat{y}^2+\hat{z}^2)+2\cdot d(\hat{y}^2+\hat{z}^2-y\hat{y}-z\hat{z}).\end{equation}
We convexify the bilinear function $2d \cdot yz$ with $c \in \R_{-}$ analogously.
We denote the above function by $h_{\textrm{aff}}(s,\act,s')$, which is affine in $y$ and $z$.
After the convexification step, we replace~\eqref{eq:probcomputation_mdp} by
\begin{align}
&p_s \geq  \sum_{s' \in S}  \big( h_{\textrm{cvx}}(s,\act,s') + h_{\textrm{aff}}(s,\act,s')\big),\nonumber\\
&\forall s \in S \setminus T,\, \forall \act \in \ActS(s), \label{Convex:constraint} 
\end{align}
which is convex in $\probmdp(s,\act,s')$, and $p_{s'}$.
The construction is similar for negative values of $d$.
The following proposition clarifies the relationship between the nonconvex and convexified QCQP.\looseness=-1

\begin{proposition}\label{prop:over_approximation}
	A feasible solution to the nonconvex QCQP in~\eqref{eq:min_mdp}--\eqref{eq:probcomputation_mdp} is also feasible to the convexified QCQP in~\eqref{eq:min_mdp}--\eqref{eq:probthreshold_mdp} together with~\eqref{Convex:constraint}.
\end{proposition}

\begin{proof}
	For any concave function $f$, its first-order approximation is a global over-estimator, i.e., for $\mathbf{v}, \mathbf{v}' \in V$:
	\begin{align}
	f(\mathbf{v})\leq f(\mathbf{v}') + \nabla f(\mathbf{v}')^{\top}(\mathbf{v}'-\mathbf{v}),
	\end{align}
	leading to $h_{\textrm{ccv}}(s,\act,s')\leq h_{\textrm{aff}}(s,\act,s')$. Therefore, by construction, any feasible solution that satisfies~\eqref{eq:probcomputation_mdp} also satisfies~\eqref{Convex:constraint}. The claim follows, as all other constraints coincide.\looseness=-1
	\end{proof}

The stricter set of constraints is often too strict and may not have any feasible solution.
We add a \emph{penalty variable} $k_{s}$ for all $s\in S \setminus T$ to all convexified constraints, which guarantees that the DC problem is feasible.
These variables allow us to measure the ``amount'' of infeasibility.
The larger the assigned value for a penalty variable, the larger the violation.
We then seek to minimize the violation of the original DC constraints by minimizing the sum of the penalty variables. 
The resulting convexified problem with the penalty variables is given by
\begin{align}
\text{minimize} &\quad p_{\sinit}+ \tau\sum_{\forall s\in S\setminus T}k_{s}\label{eq:min_qcqp_ccp}\\
\text{subject to}\nonumber\\
p_s=1, &\quad \forall s\in T,	\label{eq:targetprob_ccp}\\
 \probmdp(s,\act,s')\geq \epsgraph,&\quad\forall s,s'\in S\setminus T, \forall\act\in\ActS(s), \label{eq:well-defined_probs_ccp}\\
\sum_{s'\in S}\probmdp(s,\act,s')=1,&\quad\forall s\in S\setminus T,\forall \act \in \ActS(s),	 \label{eq:well-defined_probs_ccp1}\\
\lambda \geq p_{\sinit},&\quad 
\label{eq:probthreshold_ccp}\\
 k_{s} + p_s \geq   \sum_{s' \in S}\big( h_{\textrm{cvx}}(s,&\act,s')+h_{\textrm{aff}}(s,\act,s')\big),\label{eq:probcomputation_ccp}\\
	&\quad \forall s\in S\setminus T, \forall \act \in \ActS(s), \nonumber\\
k_{s} \geq 0,&\quad\forall s\in S\setminus T, \label{eq:slackvariable_ccp}
\end{align}
where $\tau > 0$ is a fixed \emph{penalty parameter}.
This convexified DC problem is, in fact, a convex QCQP. 
We use the constraint~\eqref{eq:probthreshold_ccp}, which is the same as constraint~\eqref{eq:probthreshold_mdp}, to ensure that we compute a solution at each iteration to minimize the violation of the convexified constraints, and guide the methods to a feasible solution with respect to the specification $\varphi$. We show an example of the convexified problem in Example~\ref{ex:convexified_dc}.
\begin{example}\label{ex:convexified_dc}
	Recall the pMC in Fig.~\ref{fig:pmc_reform} and the DC problem from Example~\ref{ex:dc}. 
	We introduce the penalty variables $k_{s_i}$ and assume a fixed $\tau$.
	We convexify around $\hat v$, and the probability values of each state, $\hat{p}_{s_0}, \hat{p}_{s_1}, \hat{p}_{s_2}$. The resulting convex problem with $d=0.5$ is given by
	\begin{align*}
	&\textnormal{minimize} \quad p_{s_0}+\tau\sum_{i=0}^2 k_{s_i} \\
	&	\textnormal{subject to}  \;\; p_{s_3}=1, \; \lambda \geq p_{s_0},\\
	&\;  k_{s_0}+p_{s_0} \geq d \cdot (v+p_{s_1})^2-2d\cdot(\hat{v}^2+\hat{p}^2_{s_1}-v\hat{v}-{p}_{s_1}\hat{p}_{s_1}),\\
	&\; k_{s_1}+p_{s_1} \geq p_{s_2}+ d \cdot (v-p_{s_2})^2-\\
	&\qquad\qquad\qquad 2d\cdot(\hat{v}^2+\hat{p}^2_{s_2}-v\hat{v}-{p}_{s_2}\hat{p}_{s_2}),\\
	&\;  k_{s_2}+p_{s_2} \geq d \cdot (v+p_{s_3})^2-2d\cdot(\hat{v}^2+\hat{p}^2_{s_3}-v\hat{v}-{p}_{s_3}\hat{p}_{s_3}),\\
	&\; 1-\epsgraph \geq v\geq \epsgraph, \\
	&\; k_{s_0}\geq 0, k_{s_1}\geq 0, k_{s_2} \geq 0. 
	\end{align*}
\end{example}
If all penalty variables are assigned to zero, we can terminate the algorithm immediately, which we state in Theorem~\ref{thm:main}.
\begin{theorem}
	\label{thm:main}
	A feasible solution of the convex DC problem in~\eqref{eq:min_qcqp_ccp} -- \eqref{eq:slackvariable_ccp}
	\begin{align*}
	\text{with}\qquad  \tau\sum_{\forall s\in S\setminus T}k_{s}=0		
	\end{align*}
	is a feasible solution to Problem 1.
\end{theorem}

\begin{proof}
	Let $\mathbf{v}$ be a feasible solution as in the theorem statement. 
	With the given condition, the constraints in~\eqref{eq:min_qcqp_ccp} -- \eqref{eq:slackvariable_ccp} reduce to the constraints in~\eqref{eq:min_mdp}--\eqref{eq:probthreshold_mdp} and~\eqref{Convex:constraint}.
	From Proposition~\ref{prop:over_approximation}, we conclude that the feasible solution is also feasible to the QCQP in~\eqref{eq:min_mdp}--\eqref{eq:probcomputation_mdp}.
	Finally, using Proposition~\ref{prob:pmdpsyn}, we conclude that the feasible solution is also a solution to Problem~1.
	We provide another proof in the Appendix~A.
	\end{proof}

%
We update the penalty parameter $\tau$ by $\mu + \tau$ for a $\mu > 0$, if any of the penalty variables are positive.
The penalty parameter is updated until an upper limit for $\tau_{\mathit{max}}$ is reached to avoid numerical problems.
Then, we convexify the $h_{\textrm{ccv}}$ functions around the current (not feasible) solution and solve the resulting convex problem. 
We repeat this procedure until we find a feasible solution, or it converges. 
The procedure may be restarted with a different initial value if the procedure converges to an infeasible solution.

\subsection{Efficiency Improvements in CCP}\label{sec:efficiencyimprovements_ccp}

In this section, we consider problem-specific implementation details and efficiency improvements for the proposed CCP method.
Specifically, we provide details on the encoding, and how we update the variables between the iterations compared to the standard CCP methods.

\paragraph{Algorithmic Improvements} We list three key improvements that we make as opposed to a naive implementation of the approaches.
\begin{inparaenum}[(1)]
	\item 	
	We efficiently precompute the states $s \in S$ that reach target states with probability $0$ or $1$ with graph algorithms~\cite{BK08}, which simplifies the QCQP in~\eqref{eq:min_mdp} -- \eqref{eq:probcomputation_mdp}.
	\item 
	Often, all instantiations with admissible parameter values yield well-defined MDPs. We verify this property via an easy preprocessing.
	Then, we omit the constraints~\eqref{eq:well-defined_probs_mdp} for the well-definedness.
	\item Parts of the encoding are untouched over multiple CCP iterations.
	Instead of rebuilding the encoding, we only update constraints that contain iteration-dependent values. 
	The update is based on a preprocessed representation of the model.
	The improvement is two-fold: We spend less time constructing the encoding, and the solver reuses previous results, making solving up to three times faster.\looseness=-1
\end{inparaenum}

\paragraph{Integrating Model Checking with CCP}

\begin{figure}
	\centering
	\input{figures/qcqpmcloop}
	\caption{CCP with model checking in the loop. 
		After each iteration, we model check the instantiated MDP to determine whether the specification is satisfied. 
		If the MDP satisfies the specification, we return the parameter instantiation. 
		Otherwise, we update CCP with the model checking results and convexify around the new solution until the procedure converges or we find a feasible solution.}
	\label{fig:ccpwithmc}
\end{figure}
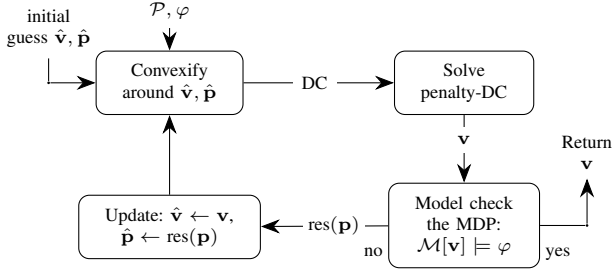

The first and foremost assumption, backed by numerical examples, is that model checking single instantiations of a pMDP is much faster than an iteration of the CCP method.
Consequently, we slightly change the loop from Fig.~\ref{fig:ccploop} to the loop in Fig.~\ref{fig:ccpwithmc}, making two important changes.
The first change is that we no longer check the penalty of a solution.
Instead, we use the values $\hat{\mathbf{v}}$, which gives rise to a parameter valuation.
Model checking at these instantiations has two benefits:
First, it allows for \emph{early termination}. 
We verify whether $\mdp[\hat{\mathbf{v}}]$ satisfies $\varphi$: 
If yes, we have found a solution even though the penalty variables have not converged to zero.
Second, the model checking procedures are numerically more stable and allow for exact arithmetic. 
Thereby, the solutions are more reliable than the solutions obtained from the convex solver.

The second change is in the update procedure.
Using the model checking results often overcomes convergence to infeasible solutions.
These problems may be described as follows:
Instead of instantiating the initial probability value for the next iteration as the obtained solution from the penalty-DC problem, we use the model checking result $\textrm{res}(\mathbf{p})$ of the instantiated MDP $\mdp[\hat{\mathbf{v}}]$ and set $\hat{\mathbf{p}}$ as the probability values of satisfying the specification for the instantiated MDP.
Model checking ensures that the probability variables are consistent with the parameter variables, i.e., that the constraints in~\eqref{eq:probcomputation_mdp} are all satisfied.\looseness=-1


%% file: figures/qcqpnaiveloop.tex
\begin{tikzpicture}[every node/.style={font=\scriptsize},>={Stealth[scale=1.5]}]
\node[draw,rectangle, inner sep=4pt, rounded corners] (conv) {\begin{tabular}{c}Convexify\\around $ \hat{\mathbf{v}}, \hat{\mathbf{p}}$\end{tabular}};
\node[circle,left=0.6cm of conv,inner sep=0pt, fill=black] (lconv) {};
\node[above=0.3cm of lconv] (start) {\begin{tabular}{c}Initial  guess\\ $ \hat{\mathbf{v}}, \hat{\mathbf{p}}$\end{tabular}};
\node[draw,rectangle, right=1.4cm of conv, inner sep=4pt, rounded corners] (dc) {\begin{tabular}{c}Solve \\ penalty-DC \end{tabular}};
\node[draw,rectangle, below=0.9cm of dc, inner sep=4pt, rounded corners] (check) {Penalty = 0?};
\node[draw,rectangle, below=0.73cm of conv, inner sep=4pt, rounded corners] (update) {\begin{tabular}{c}Update: $\hat{\mathbf{v}} \leftarrow  {\mathbf{v}}$\\
	$\hat{\mathbf{p}} \leftarrow  {\mathbf{p}}$\end{tabular}};
\node[above=0.3cm of conv] (input) {$\pmdp,\varphi$};
\node[circle,right=0.6cm of check,inner sep=0pt, fill=black] (solution-circ) {};
\node[above=0.6cm of solution-circ, inner sep=1pt] (solution) {\begin{tabular}{c}Return \\$\mathbf{v} $\end{tabular}};

\draw[->] (lconv) -- (conv);
\draw[-] (start) -- (lconv);
\draw[->] (input) -- (conv);
\draw[->] (conv) -- node[above] {} (dc);
\draw[->] (dc) -- node[above] {} (check);
\draw[-] (check) --  (solution-circ);
\draw[->] (solution-circ) --  (solution);
\draw[->] (check) --  (update);
\draw[->] (update) --  (conv);
\node[fill=white,rectangle,right=0.32cm of conv,align=center] (i) {DC};
\node[fill=none,rectangle,below right=-0.30cm and -0.05cm of check,align=center] (i) {yes};
\node[fill=none,rectangle,below left=-0.30cm and -0.05cm of check,align=center] (i) {no};
\node[fill=white,rectangle,below=0.15cm of dc,align=center] (j) {${\mathbf{v}}, {\mathbf{p}}$};

\end{tikzpicture}

%% file: figures/qcqpmcloop.tex
\begin{tikzpicture}[every node/.style={font=\scriptsize},>={Stealth[scale=1.5]}]
%

\node[draw,rectangle, inner sep=4pt, rounded corners] (conv) {\begin{tabular}{c}Convexify\\around $\hat{\mathbf{v}}, \hat{\mathbf{p}}$\end{tabular}};
\node[circle,left=0.6cm of conv,inner sep=0pt, fill=black] (lconv) {};
\node[above=0.3cm of lconv] (start) {\begin{tabular}{c}initial \\ guess $\hat{\mathbf{v}}, \hat{\mathbf{p}}$\end{tabular}};
\node[draw,rectangle, right=2.0cm of conv, inner sep=4pt, rounded corners] (dc) {\begin{tabular}{c}Solve \\ penalty-DC \end{tabular}};
\node[draw,rectangle, below=0.9cm of dc, inner sep=4pt, rounded corners] (check) {\begin{tabular}{c}Model check\\ the MDP:\\$\mdp[{\mathbf{v}}] \models \varphi$\end{tabular}};
\node[draw,rectangle, below=1.06cm of conv, inner sep=4pt, rounded corners] (update) {\begin{tabular}{c}Update: $\hat{\mathbf{v}}\leftarrow {\mathbf{v}}$,\\$\hat{\mathbf{p}}  \leftarrow \textrm{res}(\mathbf{p})$\end{tabular}};
\node[above=0.3cm of conv] (input) {$\pmdp,\varphi$};
\node[circle,right=0.6cm of check,inner sep=0pt, fill=black] (solution-circ) {};
\node[above=0.6cm of solution-circ, inner sep=1pt] (solution) {\begin{tabular}{c}Return \\$\mathbf{v} $\end{tabular}};

\draw[->] (lconv) -- (conv);
\draw[-] (start) -- (lconv);
\draw[->] (input) -- (conv);
\draw[->] (conv) -- node[above] {} (dc);
\draw[->] (dc) -- node[above] {} (check);
\draw[-] (check) --  (solution-circ);
\draw[->] (solution-circ) --  (solution);
\draw[->] (check) --  (update);
\draw[->] (update) --  (conv);
\node[fill=white,rectangle,right=0.65cm of conv,align=center] (i) {DC};
\node[fill=white,rectangle,inner sep=2pt,below left=-0.77cm and 0.30cm of check,align=center] (j) {$\textrm{res}(\mathbf{p})$};
\node[fill=white,rectangle,below=0.15cm of dc,align=center] (j) {${\mathbf{v}}$};
\node[fill=none,rectangle,below right=-0.40cm and -0.05cm of check,align=center] (i) {yes};
\node[fill=none,rectangle,below left=-0.40cm and -0.05cm of check,align=center] (i) {no};

\end{tikzpicture}

%% file: affine_trust_region.tex
\section{Sequential Convex Programming}
\label{sec:scp}
In this section, we discuss our second method, which is a sequential convex programming (SCP) approach with trust region constraints~\cite{yuan2015recent,mao2018successive,chen2013optimality}.
Similar to the penalty CCP method, the SCP method computes a locally optimal solution by iteratively approximating a nonconvex optimization problem.
The approximate problem is a linear program (LP), instead of a QCQP. 
This comes at the cost of a generally coarser approximation.\looseness=-1
%
%

More precisely, the main differences between SCP and CCP are, 
(1) the resulting convex problem is an LP, and solving is thus generally faster than solving a similar-sized QCQP, and 
(2) the convexified functions in the constraints are no longer upper bounds of the original functions. 
The approximation may generate optimal solutions in the convexified problem that are infeasible in the original problem.
Therefore, we include \emph{trust regions} and an additional model checking step similar to the CCP method to ensure that the new solution improves the objective.
The trust regions ensure that the resulting LP accurately approximates the nonconvex QCQP.
If the new solution indeed improves the objective, we accept and update the assignment of the variables and enlarge the trust region.
Otherwise, we contract the trust region, and do not update the assignment of the variables.
\subsection{Constructing the affine approximation}
We now detail how we linearize the bilinear functions in the constraints in~\eqref{eq:probcomputation_mdp}, similar to Section~\ref{sec:ccp}.
Recall that this constraint appears as
\begin{align*}
		p_s \geq \sum_{s'\in S}	\probmdp(s,\act,s')\cdot p_{s'}	&\quad 	\forall s\in S\setminus T,\,\forall \act\in\ActS(s).
\end{align*}

Similar to the previous section, consider the bilinear function in the above constraint
\begin{equation}
h(s,\act,s')=\probmdp(s,\act,s')\cdot p_{s'}
\end{equation}
and let 
\begin{equation}
\probmdp(s,\act,s')=2d\cdot y+c, \text{ and } p_{s'}=z,
\end{equation} where $y$ is the parameter variable, $z$ is the probability, and $c, d$ are constants, similar to the previous section. We then convexify $h(s,\act,s')$ as
\begin{equation}
h_{\textrm{a}}(s,\act,s') \colonequals
2 d\cdot((\hat{y}+\hat{z})+\hat{y}\cdot(z-\hat{z})+\hat{z}\cdot(y-\hat{y}))+c\cdot z,
\end{equation} where $\langle \hat{y}, \hat{z}\rangle$ are any assignments to $y$ and $z$.
Note that the function $h_{\textrm{a}}(s,\act,s')$ is  affine in the parameter variable $y$ and the probability variable $z$. 
After the linearization, the set of constraints~\eqref{eq:probcomputation_mdp} is replaced by the convex constraints
\begin{align}
 &\quad\forall s \in S \setminus T,\, \forall \act \in \ActS(s),\nonumber p_s \geq  \sum_{s' \in S} h_{\textrm{a}}(s,\act,s'), 
\end{align}
	
\begin{remark}
If the pMDP is not affine, i.e., $\probmdp(s,\act,s')$ is not affine in $V$ for every $s,s'\in S$ and $\act\in\ActS(s)$, then $h(s,\act,s')$ will not be a quadratic function in $V$ and probability variables $p_s'$.
	In this case, we can compute $h_a(s,\act,s')$ by computing a first order approximation with respect to $V$ and $p_{s'}$ around the previous assignment.
\end{remark}



Similar to the CCP method, we use \emph{penalty variables} $k_{s}$ for all $s\in S \setminus T$ to all linearized constraints, ensuring that they are always feasible.
However, the functions in these constraints do not over-approximate the functions in the original constraints. 
Therefore, a feasible solution to the linearized problem is potentially infeasible to the parameter synthesis problem.
To make sure that the linearized problem accurately approximates the parameter synthesis problem, we use a trust region constraint around the previous parameter instantiations.
The resulting LP is:%
\begin{align}
			\textnormal{minimize} &\quad p_{\sinit}+ \tau\sum_{\forall s\in S\setminus T}k_{s}\label{eq:min_qcqp_scp}\\
			\textnormal{subject to}\nonumber\\
		 p_s=1,	 &\quad\forall s\in  T,	\label{eq:targetprob_scp}\\
	 \probmdp(s,\act,s')\geq \epsgraph,		&\quad\forall s,s'\in S\setminus T, \forall\act\in\ActS(s),  \label{eq:well-defined_probs_scp}\\
	 \sum_{s'\in  S}\probmdp(s,\act,s')=1,	&\quad	\forall s\in S,\forall \act \in \ActS(s), \label{eq:well-defined_probs_scp1}\\
		\lambda \geq p_{\sinit},	&\quad 
								 \label{eq:probthreshold_scp}\\
k_{s} + p_s \geq  \sum_{s' \in \setminus T}   h_{\textrm{a}}(s,\act,s'),	&\quad \forall s\in S\setminus T, \forall \act \in \ActS(s),	\label{eq:probcomputation_scp}\\
			 k_{s} \geq 0,&\quad\forall s\in S\setminus T. 
			 \label{eq:slackvariable_scp}\\
		  \nicefrac{\hat{p}_{s}}{\delta'}\leq p_s\leq \hat{p}_s \cdot \delta'	& \quad  \forall s \in S\setminus T,\label{eq:trust_region_scp},\\
		  	\quad\nicefrac{\hat{\probmdp}(s,\act,s')}{\delta'} \leq \probmdp(s,\act,s')\leq \hat{\probmdp}&(s,\act,s')\cdot \delta',\label{eq:trust_region_scp2}\\
			&\quad	\forall s,s'\in S\setminus T,\forall \act \in \ActS(s), \nonumber	
		\end{align}
where $\tau > 0$ is as defined before, and $\hat{\probmdp}(s,\act,s')$ and $\hat{p}_s$ denotes the previous assignment for the parameter and probability variables. 
The constraints~\eqref{eq:trust_region_scp}--\eqref{eq:trust_region_scp2} are the trust region constraints. $\delta > 0$ is the size of the trust region, and $\delta'=\delta+1$. 

Unlike the convexification in the CCP method, the linearization step in the SCP does not over approximate functions in the constraint~\eqref{eq:probcomputation_scp}. 
Therefore, we cannot provide the soundness guarantees for the solutions to the above LP as we did for the CCP method in Theorem~\ref{thm:main}. 
On the other hand, the trust region constraints ensure that the linearization is accurate by restricting the set of feasible solutions around the previous solution. 
Recall that we integrate the model-checking step into the SCP method to obtain soundness for the SCP method.
We also show in the examples by a comparison with another SCP method that integrating the model-checking step significantly improves the performance.
Similar to the CCP method, we use the constraint~\eqref{eq:probthreshold_scp} to minimize the violation of the convexified constraints.
We demonstrate the linearization in Example~\ref{ex:scp}.


\begin{example}
	\label{ex:scp}
	Recall the pMC in Fig.~\ref{fig:pmc_reform} and the QCQP from Example~\ref{ex:nlp}.  After linearizing around an assignment for $\hat{v}, \hat{p}_{s_0}, \hat{p}_{s_1},$ and $\hat{p}_{s_2}$, the resulting LP with a trust region radius $\delta>0$ is given by
	\begin{align*}
		&\textnormal{minimize } \quad p_{s_0}+\tau\sum_{i=0}^2 k_{s_i} \\
&	\textnormal{subject to}  \nonumber\\
&\; p_{s_3}=1,\; \lambda \geq p_{s_0},\\
&\;  k_{s_0}+p_{s_0} \geq \hat{v}\cdot\hat{p}_{s_1}+\hat{p}_{s_1}\cdot(v-\hat{v})+\hat{v}\cdot(p_{s_1}-\hat{p}_{s_1}),\\
&\; k_{s_1}+p_{s_1} \geq p_{s_2}-\hat{v}\cdot\hat{p}_{s_2}-\hat{p}_{s_2}\cdot(v-\hat{v})-\hat{v}\cdot(p_{s_2}-\hat{p}_{s_2}),\\
&\;  k_{s_2}+p_{s_2} \geq\hat{v}\cdot\hat{p}_{s_3}+\hat{p}_{s_3}\cdot(v-\hat{v})+\hat{v}\cdot(p_{s_3}-\hat{p}_{s_3}),\\
& \; k_{s_0}\geq 0, k_{s_1}\geq 0, k_{s_2} \geq 0, \\
& \; \nicefrac{\hat{p}_{s_0}}{\delta'}\geq  p_{s_0}\geq \hat{p}_{s_0}\cdot\delta',
\; \nicefrac{\hat{p}_{s_1}}{\delta'}\geq  p_{s_1}\geq \hat{p}_{s_1}\cdot\delta',
\\&\; \nicefrac{\hat{p}_{s_2}}{\delta'}\geq  p_{s_2}\geq \hat{p}_{s_2}\cdot\delta', \; \nicefrac{\hat{v}}{\delta'}\geq  v\geq \hat{v}\cdot\delta'.
	\end{align*}
\end{example}
We detail our SCP method in Fig.~\ref{fig:scpwithmc}. 
We initialize the method with a guess for the parameters $\hat{\mathbf{v}}$, for the probability variables $\hat{\mathbf{p}}$, and the trust region $\delta>0$.
Then, we solve the LP~\eqref{eq:min_qcqp_scp}--\eqref{eq:trust_region_scp} that is linearized around $\hat{\mathbf{v}}$ and probability variables $\hat{\mathbf{o}}$.

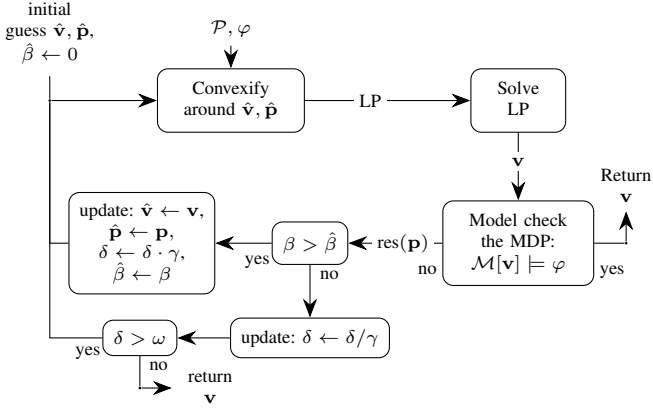
\begin{figure}[t]
	\centering
	\input{figures/scpmcloop}
	\caption{SCP with model checking in the loop. 
		The NLP~\eqref{eq:min_mdp}--\eqref{eq:probcomputation_mdp} is linearized around $\hat{\mathbf{v}}, \hat{\mathbf{p}}$.
		Then, we solve the LP~\eqref{eq:min_qcqp_scp}--\eqref{eq:trust_region_scp} an optimal solution to the parameter values, denoted by $\hat{\mathbf{v}}$.
		After each iteration, we model check the instantiated MDP $\mdp[\mathbf{v}]$ to determine whether the specification is satisfied. 
		If the instantiated MDP satisfies the specification, we return the parameter instantiation. 
		Otherwise, we check whether the reachability probability, denoted by $\beta$, is improved compared to the previous iteration, denoted by $\hat{\beta}$.
		If the probability is improved, we accept this step, update the assignment for the parameters and the probability variables, and increase the size of the trust region $\delta$ by $\gamma$.
		Otherwise, we do not update the assignment, and decrease the size of the trust region.}
	\label{fig:scpwithmc}
\end{figure}

After obtaining an instantiation to the parameters $\mathbf{v}$, we model check the instantiated MDP $\mdp[\mathbf{v}]$ to obtain the values of probability variables $\textrm{res}(\mathbf{p})$ for the instantiation $\mathbf{v}$.
If the instantiated MDP indeed satisfies the specification, we return the instantiation $\mathbf{v}$.
Otherwise, we check whether the probability of reaching the target set  $\beta$ is larger than the previous best value $\hat{\beta}$.
If $\beta$ is larger than $\hat{\beta}$, we update the values for the probability and the parameter variables, and enlarge the trust region.
Else, we reduce the size of the trust region, and resolve the problem that is linearized around $\mathbf{v}$ and $\mathbf{p}$.
This procedure is repeated until a parameter instantiation that satisfies the specification is found, or the value of $\delta$ is smaller than $\omega >0$.
The intuition behind enlarging the trust region is as follows: If the instantiation to the parameters $\mathbf{v}$ increases the probability of reaching the target set  $\beta$ over the previous solution, then we conclude that the linearization is accurate. 
Consequently, the SCP method may take a larger step in the next iteration for faster convergence in practice.

For expected cost specifications, the resulting algorithm is similar, except, we accept the parameter instantiation if the expected cost is reduced compared to the previous iteration, and initialize $\hat{\beta}$ with a large constant.

\subsection{Efficiency Improvements in SCP}\label{sec:efficiencyimprovements_scp}

Similar to the previous section, we consider several efficiency improvements for the SCP method. 
We also apply the algorithmic improvements in Section~\ref{sec:efficiencyimprovements_ccp} about building the encoding to the SCP method.

\paragraph{Integrating Trust Regions and Model Checking with SCP}

In this section, we discuss the implementation details for the procedure in Fig.\ref{fig:scpwithmc} and discuss how we update trust regions in our implementation. 
In nonlinear optimization, trust region algorithms obtain the new iterate point by searching in a trust region of the current iterate point~\cite{yuan2015recent,mao2018successive,chen2013optimality}.
Trust region algorithms check if the approximated problem accurately represents the nonlinear optimization problem, and adapt the size of the trust region in each iteration.
In our algorithm, we check if the convexification is accurate by checking whether the obtained parameter instantiation improves the objective value.
If we improve the objective value, i.e., the probability of satisfying the specification, we will enlarge the trust region.
Otherwise, we conclude that the size of the trust region is too large, and contract its radius.

There are two critical differences between the method in Fig.~\ref{fig:scpwithmc} and the existing SCP methods. 
First, we perform a model checking to ensure that the obtained probability variables and parameter variables satisfy the constraints in~\eqref{eq:probcomputation_mdp}.
It also provides an accurate point for the next iteration as the probability and parameter variables are consistent on the underlying pMDP.

Second, we can compute the actual objective value by model checking and determine to accept the iterate instead of checking the feasibility of the solution for each constraint and the change of an approximate objective function~\cite{bolte2016majorization,mao2018successive}.
Computing the actual objective by model checking allows us to take larger steps in each iteration and significantly improve the algorithm's performance in practice.

%% file: figures/scpmcloop.tex
\begin{tikzpicture}[every node/.style={font=\scriptsize},every text node part/.style={align=center},>={Stealth[scale=1.5]}]
\node[draw,rectangle, inner sep=4pt, rounded corners] (conv) {\begin{tabular}{c}Convexify\\around $\hat{\mathbf{v}}, \hat{\mathbf{p}}$\end{tabular}};
\node[circle,left=1.4cm of conv,inner sep=0pt, fill=black] (lconv) {};
\node[above=0.3cm of lconv] (start) {\begin{tabular}{c}initial \\ guess $\hat{\mathbf{v}}, \hat{\mathbf{p}}$,\\
	$\hat{\beta}\leftarrow 0$\end{tabular}};
\node[draw,rectangle, right=2.2cm of conv, inner sep=4pt, rounded corners] (dc) {\begin{tabular}{c}Solve \\ LP \end{tabular}};
\node[draw,rectangle, below=0.9cm of dc, inner sep=4pt, rounded corners](check) {\begin{tabular}{c}Model check\\ the MDP:\\$\mdp[\mathbf{v}] \models \varphi$\end{tabular}};
\node[draw,rectangle, left=1.25cm of check, inner sep=4pt, rounded corners](checkobj) {$\beta >\hat{\beta}$};
\node[draw,rectangle, left= 0.75cm of checkobj,  inner sep=4pt, rounded corners] (update) {update: $\hat{\mathbf{v}} \leftarrow \mathbf{v}$,\\$\hat{\mathbf{p}}\leftarrow\mathbf{p},$\\ $\delta\leftarrow\delta\cdot \gamma$,\\$\hat{\beta}\leftarrow {\beta}$};
\node[draw,rectangle, below= 0.7cm of checkobj,  inner sep=4pt, rounded corners] (update_rej) {update: $\delta\leftarrow\delta/ \gamma$};
\node[draw,rectangle, left = 0.70cm of update_rej, inner sep=4pt, rounded corners](checktrust) {$\delta >\omega$};
\node[above=0.3cm of conv] (input) {$\pmdp,\varphi$};
\node[circle,right=0.4cm of check,inner sep=0pt, fill=black] (solution-circ) {};
\node[above=0.4cm of solution-circ, inner sep=1pt] (solution) {\begin{tabular}{c}Return \\$\mathbf{v} $\end{tabular}};
\node[circle,below=0.4cm of checktrust,inner sep=0pt, fill=black] (solution_trust-circ) {};
\node[right=0.4cm of solution_trust-circ,inner sep=0pt] (solution_trust) {\begin{tabular}{c}return \\$\mathbf{v} $\end{tabular}};
\node[above right=-1.0cm and 0.03cm of update] (res) {};
\draw[-] (check) --  (solution-circ);
\draw[->] (solution-circ) --  (solution);
\draw[->] (checkobj) --(update_rej);
\draw[->] (check) --  node[pos=0.1,right,below right=-0.02cm and 0.1cm of check] {} (checkobj);
\draw[->] (checkobj) --  node[pos=0.1,right,below right=-0.35cm and -0.4cm of check] {} (update);
\draw[-] (update) -| (lconv);
\draw[->] (update_rej) -- (checktrust);
\draw [-] (checktrust)-| (lconv);
\draw [-] (checktrust)-- (solution_trust-circ);
\draw[->] (solution_trust-circ) --(solution_trust);
\node[above left=-0.30cm and 0.0cm of checktrust] {};
\node[below right=0.0cm and -0.50cm of checktrust] {};
\draw[->] (lconv) -- (conv);
\draw[-] (start) -- (lconv);
\draw[->] (input) -- (conv);
\draw[->] (conv) -- node[above] {} (dc);
\draw[->] (dc) --  (check);
\node[fill=white,rectangle,inner sep=2pt,below left=-0.77cm and 0.1cm of check,align=center] (j) {$\textrm{res}(\mathbf{p})$};
\node[fill=white,rectangle,below=0.25cm of dc,align=center,inner sep=2pt] (j) {$\mathbf{v}$};
\node[fill=white,rectangle,right=0.65cm of conv,align=center,inner sep=2pt] (i) {LP};
\node[fill=none,rectangle,below left=-0.20cm and -0.05cm  of checkobj,align=center,inner sep=2pt] (i3) {yes};
\node[fill=none,rectangle,below right=-0.00cm and -0.45cm of checkobj,align=center,inner sep=2pt] (i4) {no};
\node[fill=none,rectangle,below left=-0.20cm and -0.05cm  of checktrust,align=center,inner sep=2pt] (i3) {\added{yes}};
\node[fill=none,rectangle,below right=-0.00cm and -0.45cm of checktrust,align=center,inner sep=2pt] (i4) {\added{no}};
\node[fill=none,rectangle,below right=-0.40cm and -0.05cm of check,align=center] (i5) {yes};
\node[fill=none,rectangle,below left=-0.40cm and -0.05cm of check,align=center] (i6) {no};
\end{tikzpicture}

%% file: convergence.tex
\section{Convergence Properties of the Proposed Methods}
\label{sec:convergence}
In this section, we discuss the convergence properties of the proposed methods.
We also discuss the conditions on the constrained and penalty problem having the same set of locally optimal solutions.

\paragraph{Convergence properties of CCP}

The convergence of CCP is discussed in~\cite{lipp2016variations,yuille2002concave}. 
If CCP is started with a feasible point, then all of the solutions in each iteration will be feasible.
Additionally, the objective will decrease monotonically and will converge.
Reference~\cite{lanckriet2009convergence} showed that CCP converges to a solution that satisfies the full KKT conditions, which are necessary conditions for a solution to be locally optimal.
The convergence rate of CCP is established in~\cite{khamaru2018convergence} if the feasible set is convex, which is not the case for the parameter synthesis problem.
Therefore, we are not aware of any convergence rate results for the CCP method in the parameter synthesis problem.

\paragraph{Convergence properties of SCP}


The convergence rate statements of the trust region methods~\cite{yuan2015recent,mao2018successive,chen2013optimality} and other SCP methods~\cite{bolte2016majorization,auslender2013extended} rely on regularization assumptions such as Lipschitz continuity for the gradients of the functions in the objective and the constraints.
The QCQP in~\eqref{eq:min_mdp}--\eqref{eq:probcomputation_mdp} satisfies the regularization assumption, as all the functions in the objective and constraints are quadratic.\
The critical step in the convergence proofs is to show the existence of a small enough trust region to approximate the nonlinear optimization problem and sufficiently decrease the objective function in each iteration.
Recent linear and superlinear convergence results~\cite{bolte2016majorization,mao2018successive} rely on results from dynamical systems~\cite{kurdyka1998gradients}  to show linear convergence of iterative methods with a real-analytic objective and constraint functions.
We state a convergence property of the SCP method described by Fig.~\ref{fig:scpwithmc}. 

\begin{proposition}
	The sequence $(\mathbf{v}, \mathbf{p})$ generated by SCP method in Fig.~\ref{fig:scpwithmc} converges to a limit point.
\end{proposition}
\begin{proof}
	Let $f^k(\mathbf{v}, \mathbf{p})$ be the reachability probability of the instantiated MDP $\mdp[\mathbf{v}]$ at the $k-$th iteration. 
	It is clear that this sequence is bounded below by $0$, and above by $1$, and is monotonically increasing as we accept the parameter solutions only if the reachability probability is improved. 
	Therefore, due to Lebesgue's monotone convergence theorem~\cite[11.28]{rudin1976principles}, this sequence is convergent for reachability probabilities.
	For expected cost specifications, the sequence is monotonically decreasing, and bounded above by the instantiated cost in the first iteration, and below by $0$; therefore it is convergent.
\end{proof}

The model checking step in Fig.~\ref{fig:scpwithmc} involves a nonsmooth mapping from the parameter instantiations to probability variables due to maximizing the reachability probability at each action and makes the convergence rate analysis difficult.
For completeness, we show the convergence rate of a similar iterative method without model checking in Appendix~B.
The local convergence rate of the SCP method is \emph{linear} in Appendix~B.
However, in our numerical examples, we observed that the SCP method in Appendix~B does not perform well in practice due to requiring a small enough step size to ensure convergence.
On the other hand, we can check if we improved the objective directly in Step~2 of Algorithm~1, allowing us to take larger step sizes and reduce the required number of iterations compared to the existing methods.

\paragraph{Justification of penalty functions}
\label{sec:penalty}

In this section, we discuss the merit of using penalty functions for the CCP and SCP methods. 
Using penalty functions~\cite{zhang1985improved,mayne1987exact,han1979exact} ensure that each convexified problem is feasible. 
A \emph{penalty problem}~\cite{mao2018successive} puts a penalty for violating each constraint in a constrained optimization problem instead of enforcing them as hard constraints.
A penalty function is exact if a constrained optimization problem and the penalty problem have the same optimality conditions.

For simplicity, we denote our variable pair $(\mathbf{v}, \mathbf{p})$  as $x$ in this section.
Following the definitions in~\cite[Section 5.1.1]{boyd_convex_optimization}, we denote objective in~\eqref{eq:min_mdp} as $f_0(x)$, the inequality constraints in~\eqref{eq:well-defined_probs_mdp}, \eqref{eq:probthreshold_mdp}--\eqref{eq:probcomputation_mdp} as 
\begin{align*}
f_i(x)\leq 0, \quad i=1,\ldots,m,
\end{align*}
where $m$ is the number of inequality constraints,  and the equality constraints in~\eqref{eq:targetprob_mdp} and~\eqref{eq:well-defined_probs_mdp1} as
\begin{align*}
g_i(x)=0, \quad i=1,\ldots,p,
\end{align*}
where $p$ is the number of inequality constraints. We now state the stationary part of the first order necessary conditions for a point~$\hat{x}$ to be a locally optimal solution.
\begin{theorem}[Karush-Kuhn-Tucker conditions,~\cite{boyd_convex_optimization}]
	If $\bar{x}$ is a locally optimal solution to the NLP in~\eqref{eq:min_mdp}--\eqref{eq:probcomputation_mdp}, then there exist optimal Lagrange multipliers $\mu^{\star}_i \geq 0, i=1,\ldots,m$ and $\nu^{\star}_i, i=1\ldots,p$ that satisfies the following
	\begin{align*}
	\nabla f_o(\bar{x})+\sum_{i=1}^{m}\mu^{\star}_i\nabla f_i(\bar{x})+\sum_{i=1}^{p}\nu^{\star}_i\nabla g_i(\bar{x})=0,
	\end{align*}
	under an appropriate constraint qualifier condition.
	\label{thm:kkt}
\end{theorem}
	
\begin{proof}
We provide the proof for satisfying the linear independence constraint qualification in Appendix B.
\end{proof}
For the penalty problem, let the objective be
\begin{align*}
f_0(x)+\sum_{i=1}^{m}\bar{\mu}_i\text{max}(0, f_i(x))+\sum_{i=1}^{p}\bar{\nu}_i\vert g_i(x)\vert,
\end{align*}
where $\bar{\mu}_i, i=1,\ldots,m$ and $\bar{\nu}_i, i=1\ldots,p$ are the penalty weights and the penalty problem has no constraints. 
We now state the \emph{exactness} condition, which states the necessary and sufficient condition on the original QCQP and the penalty problem having the same set of locally optimal solutions.

\begin{theorem}[Exactness conditions,~\cite{han1979exact}]
	If $\bar{x}$ satisfies the condition in Theorem~\ref{thm:kkt} with Lagrange multipliers $\mu^{\star}_i \geq 0, i=1,\ldots,m$ and $\nu^{\star}_i, i=1\ldots,p$, and let the penalty weights $\bar{\mu}_i, i=1,\ldots,m$ and $\bar{\nu}_i, i=1\ldots,p$ satisfy
	\begin{align*}
	&\bar{\mu}_i > |\mu^{\star}_i|, \quad i=1,\ldots,m,\\
	&\bar{\nu}_i > |\nu^{\star}_i|, \quad i=1,\ldots,p,
	\end{align*}
	then $\bar{x}$ is also a locally optimal solution for the penalty problem. The converse result also holds.
	\label{thm:exact}
\end{theorem}
Theorem~\ref{thm:exact} guarantees the existence of penalty weights that preserve the set of locally optimal solutions.
In practice, we select a large enough $\tau$ in SCP and CCP to minimize the violations of the constraints.

%% file: experiments.tex
\section{Numerical Examples}\label{sec:experiments}
%
We evaluate our approaches on benchmark examples that are subject to either reachability or expected cost specifications.
We implement the proposed CCP and SCP methods with the discussed efficiency improvements from Section~\ref{sec:efficiencyimprovements_ccp} and Section~\ref{sec:efficiencyimprovements_scp}  in the parameter synthesis framework \prophesy~\cite{dehnert-et-al-cav-2015}.
We use the probabilistic model checker \storm~\cite{DBLP:conf/cav/DehnertJK017} to extract an explicit representation of a pMDP.
We keep the pMDP in memory and update the parameter instantiations using an efficient data structure to enable efficient repetitive model checking.
To solve convex QCQPs and LPs, we use Gurobi 9.0~\cite{gurobi}, configured for numerical stability.

\paragraph{Environment}
We evaluate the SCP and CCP methods with 16 3.1 GHz cores, a 32 GB memory limit, and an 1200 seconds time limit (TO).
The objective is to find a feasible parameter valuation for pMCs and pMDPs with specified thresholds on probabilities or costs in the specifications, as in Problem~\ref{prob:pmdpsyn}. 
We ask for a well-defined valuation of the parameters, with $\epsgraph=10^{-6}$.
We run all the approaches with the same configuration of \storm. 
For pMCs, we enable weak bisimulation, which reduces the number of states of the pMC in all examples.
We do not use bisimulation for pMDPs.
We provide all the codes and log files for the numerical examples in \texttt{https://github.com/mcubuktepe/pMDPsynthesis}.

\paragraph{Baselines}
We compare the runtimes with three baselines: a particle swarm optimization (PSO) implementation, solving the nonconvex QCQP directly with Gurobi~\cite{gurobi}, indicated by ``GRB'' in the following results, and an SCP algorithm S$l^1$QP, indicated by ``S$l^1$QP'' in the following results, from~\cite{bolte2016majorization}, with convergence guarantees to KKT points.
PSO is a heuristic sampling approach that searches the parameter space, inspired by~\cite{chen2013model}.
For each valuation, PSO performs model checking \emph{without} rebuilding the model, instead it adapts the matrix from previous valuations.
As PSO is a randomized procedure, we run it five times with random seeds 0--19. 
The PSO implementation requires the well-defined parameter regions to constitute a hyper-rectangle, as proper sampling from polygons is a nontrivial task.
For pMCs that originate from POMDPs, we also compare with a POMDP solver \tool{PRISM-POMDP}~\cite{norman2017verification} by using the relation between pMCs and POMDPs~\cite{DBLP:journals/corr/abs-1710-10294}.
We note that \tool{PRISM-POMDP} computes a lower and an upper bound on the probability of satisfying a temporal logic specification by discretizing the uncountable belief space. 
Our approach computes a finite-memory policy that maximizes or minimizes the probability of satisfying a temporal logic specification.

%

\paragraph{Tuning constants}
For the CCP method, we initialize the penalty parameter $\tau=0.05$ for reachability, and $\tau=5$ for expected cost, a conservative number in the same order of magnitude as the values $\hat{p}_s$.
As expected cost evaluations have wider ranges than probability evaluations, we observed that a larger $\tau$ increases numerical stability.
For CCP, we pick $\mu = \max_{s\in S\setminus T }\hat{p}_s$.
We update $\tau$ by adding $\mu$ after each iteration. 
Empirically, increasing $\tau$ with larger steps is beneficial for the run time but induces more numerical instability.
In contrast, in the literature, the update parameter $\mu$ is frequently used as a constant, \ie, it is not updated between the iterations. 
In, e.g,~\cite{lipp2016variations}, $\tau$ is multiplied by $\mu$ after each iteration.

For the SCP method, we use a constant $\tau=10^4$ for the reachability and expected cost specifications.
The initial trust-region value is $\delta=2$ and $\gamma=1.5$, where we adjust the size of the trust region after each iteration.
Finally, we use $\omega=10^{-4}$ and terminate the procedure if the trust region is smaller than $\omega$.\looseness=-1


Our initial guess $\hat{v}$ is the center of the parameter space and thereby minimize the worst-case distance to a feasible solution.
For $\hat{p}_s$, we use the threshold $\lambda$ from the specification $\reachProplT$ to initialize the probability variables, and analogously for expected cost specifications.

\subsection{Case Study: Satellite Collision Avoidance}
\begin{figure*}[t]
	\centering
\begin{subfigure}[t]{0.499\linewidth}
	\centering
	\includegraphics[trim = 90 0 136 0, clip, width=0.999\linewidth]{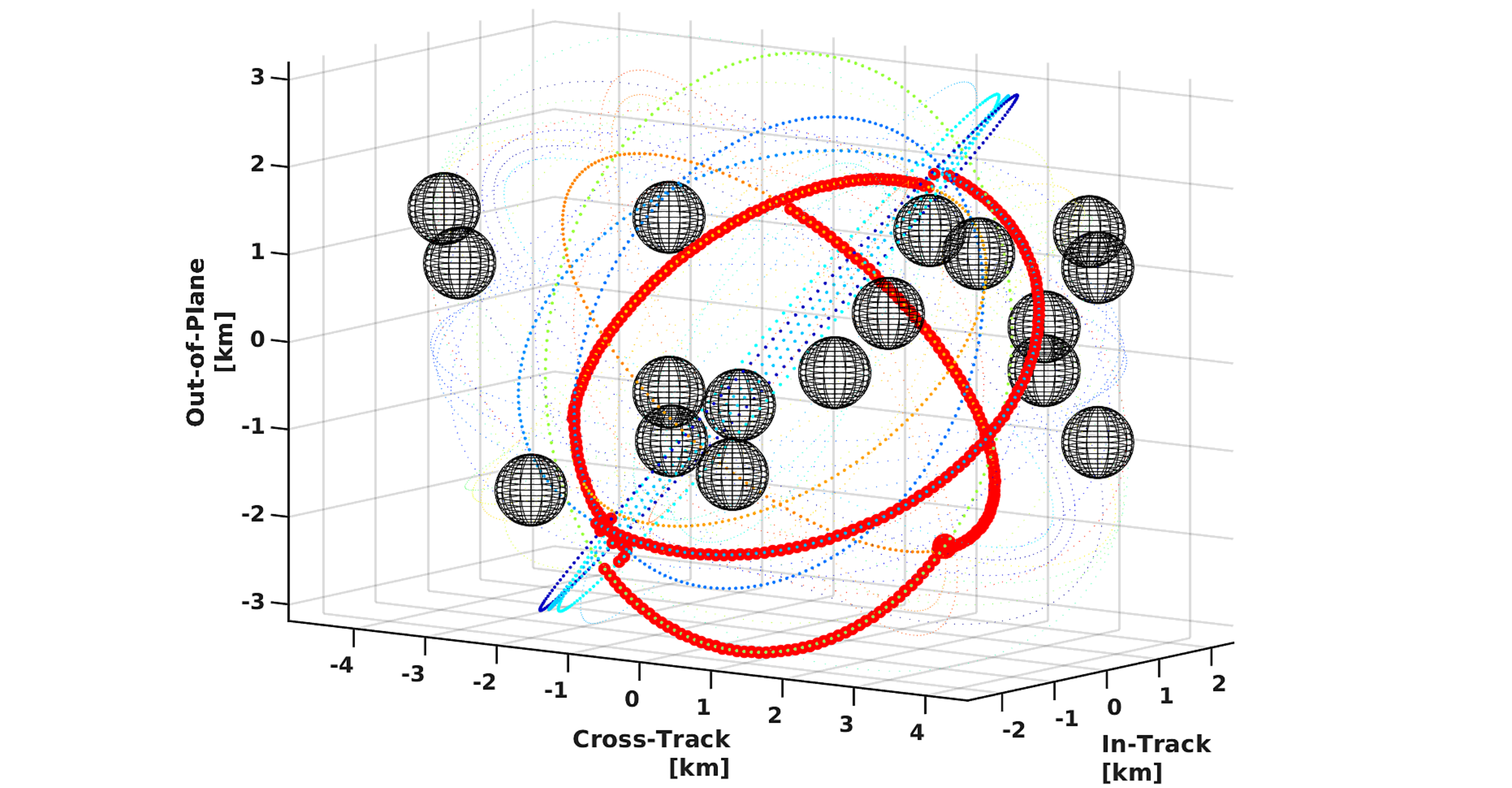}	
	\caption{Trajectory obtained from a memoryless policy.}
	\label{fig:sat1}
\end{subfigure}%
\hfill
\begin{subfigure}[t]{0.499\linewidth}
		\centering
			\includegraphics[trim = 90 0 136 0, clip, width=0.999\linewidth]{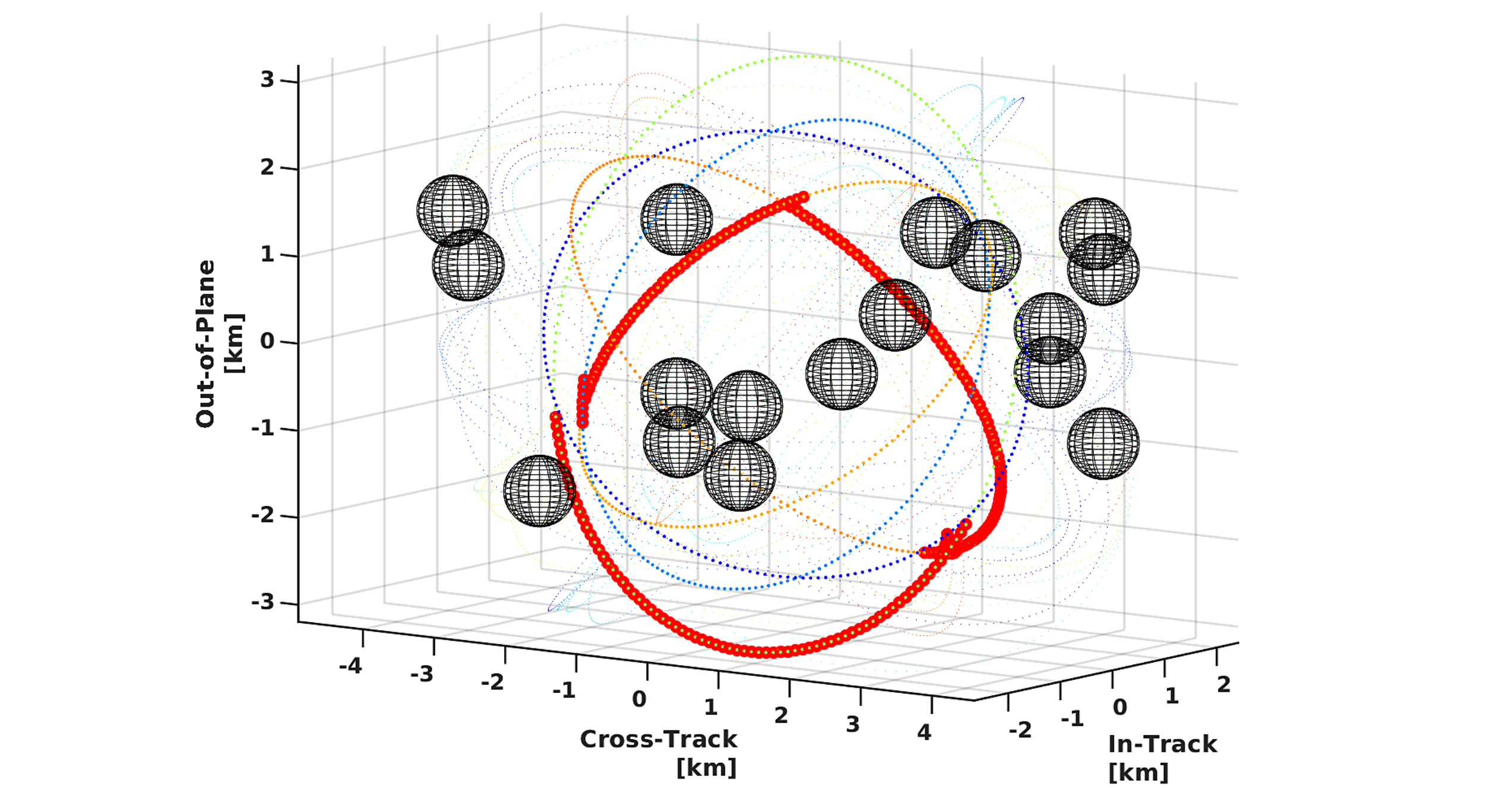}	
				\caption{Trajectory obtained from a fimite-memory policy with 5 memory of observations.}
			\label{fig:sat5}
		\end{subfigure}
	\caption{The obtained trajectories are shown in red that successfully finish an orbit around the origin. We highlight the initial location by a big circle. We highlight the NMTs used during the orbit and depict the objects to avoid by black spheres.}
\label{fig:sat}
\end{figure*}

In this example, we consider a satellite collision avoidance problem~\cite{frey2017constrained,hobbs2020taxonomy}. 
We formulate the spacecraft dynamics in Hill's reference
frame, with the origin at a specified location on a
circular spacecraft orbit~\cite{weiss2015safe,wie2008space}. 
We define the three axes of the motion as in~\cite{weiss2015safe,wie2008space}. 
We consider circular orbits in this example, as it yields time-invariant satellite dynamics with natural motion trajectories (NMTs).
The nominal orbital radius for the satellite dynamics is 7728 km.


\setlength{\tabcolsep}{5pt} 
\begin{table*}[t]
	\centering
	\footnotesize{
		\caption{Results for the satellite collision avoidance example.}
		\label{tab:results_satellite}
		\scalebox{1.00}{
			\begin{tabular}{ll|rrr|rrr|rr|rr|rr|rr}
				\multicolumn{2}{c|}{Problem} &	\multicolumn{3}{c|}{Info} & \multicolumn{3}{c|}{PSO} & \multicolumn{2}{c|}{SCP} & \multicolumn{2}{c|}{CCP} & \multicolumn{2}{c|}{S$l^1$QP} &  \multicolumn{1}{c}{GRB}\\
				Set & Spec   & States   & Trans.  & Par.      & tmin   & tmax   & \textbf{tavg}  & \textbf{t}    & iter & \textbf{t}     & iter & \textbf{t}     & iter  & \textbf{t}  \\\hline
				Satellite & $\p_{\geq 0.5}$       & 6265   & 17436  & \highlight{231}     &  --      & --       &   \TO    &  \textbf{8} & 6 & 1142 & 421 & 977 & 1061    & \TO \\
				Satellite & $\p_{\geq 0.9}$       & 6265   & 17436  & \highlight{231}     &  --      & --       &   \TO    &  \textbf{14} & 12 & \TO  & --& \TO  & --  & \TO \\
				Satellite & $\p_{\geq 0.95}$       & 6265   & 17436  & \highlight{231}     &  --      & --       &   \TO   & \TO & -- & \TO  & -- & \TO  & --  & \TO \\
				Satellite-fm & $\p_{\geq 0.95}$       & 31325   & 156924  & \highlight{2555}     &  --      & --       &   \TO    &  \textbf{146} & 10 & \MO & -- & \TO & --  & \TO \\					
				Satellite-fm & $\p_{\geq 0.995}$       & 31325   & 156924  & \highlight{2555}     &  --      & --       &   \TO    &  \textbf{239} & 18 & \MO  & -- & \TO  & --  & \TO \\		
				Satellite-1440 & $\p_{\geq 0.995}$       & 217561   & 615433  & \highlight{2248}     &  --      & --       &   \TO    &  \textbf{386} & 4 & \MO  & --  & \TO  & -- & \TO \\	
				Satellite-3600 & $\p_{\geq 0.995}$       & 217561   & 615433  & \highlight{5337}     &  --      & --       &   \TO    &  \textbf{336} & 4 & \MO   & -- & \TO  & -- & \TO \\
				Satellite-7200 & $\p_{\geq 0.995}$       & 217561   & 615433  & \highlight{10042}     &  --      & --       &   \TO    &  \textbf{370} & 4 & \MO    & -- & \TO  & --  & \TO \\
			\end{tabular}
	}}
\end{table*}

The spacecraft dynamics include periodic behavior around the earth, and a set of natural motion trajectories (NMTs)~\cite{kim2007mission,frey2017constrained}. 
In this example, we maximize the probability of successfully finishing a cycle in orbit by avoiding a collision with other objects in space.
Given a set of NMTs, we form an undirected graph with one node corresponding to each closed NMT and time index $t$, and two nodes share an edge if the distance between them is 250km.
If two nodes share an edge, then it is possible to perform a transfer between two nodes.
The state variables of the POMDP are (1) $n \in \lbrace 1, \ldots, 36\rbrace$ depicting the current NMT of the satellite, and (2) $t\in \lbrace 1, \ldots, I\rbrace$, depicting the current time index for a fixed NMT.
We use different values of $I$ in the examples. 
The satellite can follow its current NMT with no fuel usage.
The satellite can get an observation every $10$ steps, resulting in $720$ different observations.
In each NMT and time index, the satellite can either stay in the current orbit, incrementing the time index by $1$ or can transfer into a different nearby NMT if the underlying two nodes share an edge.
In our model, the probability of successfully switching to another NMT is $0.9$, and the satellite will transfer to a different nearby orbit with a probability of $0.1$.

The specification is $\p_{\geq \lambda}(\finally T)$, which asserts that the probability of finishing a cycle in orbit without colliding another object is greater than a threshold $\lambda$.
The objective is to compute a switching strategy that satisfies the specification with a probability that is greater than $\lambda$, ensuring that the satellite does not collide with another object with high probability. 
For example, in the satellite collision avoidance problem, the strategy can determine at which point in space to perform a transfer between two orbits.
This strategy and the corresponding controller for the transfer can then be implemented in the underlying system to perform collision avoidance.

%
%

In our first model, denoted by ``Satellite,'' we discretize the trajectory into $I=200$ time indices.
To further show the scalability of our SCP method, we also compute a finite-memory policy with $I=200$, denoted by ``Satellite-fm''.
We also discretize the trajectory into $I=8000$ time indices with $36$ NMTs, allowing a control input to be given to the satellite around every $0.2$ second.
We also consider the effect of having different levels of sensor quality, where the satellite can get an observation in every $40$, $80$, and $200$ time steps.
The resulting POMDPs have $288000$ states and $7200$, $3600$, and $1440$ observations, denoted by ``Satellite-7200,'' ``Satellite-3600,'' and ``Satellite-1440,'' respectively.
The resulting pMCs have $10042$, $5337$, and $2448$ parameters, respectively.

The detailed results for the models are shown in Table~\ref{tab:results_satellite}.
The first two columns refer to the benchmark instance, the next column to the specification. 
We give the number of states (States), transitions (Trans.), and parameters (Par.).
We then give the \emph{minimum} (tmin), the \emph{maximum} (tmax) and \emph{average} (\textbf{tavg}) runtime (in seconds) for PSO with different seeds, and the runtime obtained using CCP, SCP and S$l^1$QP (\textbf{t}).
We also give the number of CCP, SCP and S$l^1$QP iterations (iter).

We show the trajectories from a memoryless policy in Fig.~\ref{fig:sat1} and from a policy with finite memory in Fig.~\ref{fig:sat5}. 
The average length of the trajectory with the memoryless policy is 402, which is twice the length of the trajectory from the finite-memory policy, given by 215.
The expected cost is due to switching to a different NMT, for the finite-memory policy is 619 Newton $\cdot$ seconds. 
In contrast, it is 1249 Newton $\cdot$ seconds for the memoryless policy.
We give detailed results of the computation in Table~\ref{tab:results_satellite}.
We observe that after $6$ minutes of computation, using larger pMCs with finite-memory policy yields superior policies in the probability of satisfying the specification with $I=200$.

We obtained a solution with a collision probability of less than $10^{-3}$ on models with $I=8000$ and using memoryless policies, showing the benefit of increasing the number of time indices compared to the previous examples.
The computation took $370$, $336$, and $386$ seconds for the models with $7200$, $3600$, and $1440$ observations, respectively, showing that our approach does not scale exponentially with the number of observations and parameters.

\setlength{\tabcolsep}{5pt} 
\begin{table*}[t]
	\centering
	\footnotesize{
		\caption{pMC benchmark results}
		\label{tab:results_pmc}
		\scalebox{1.00}{
			\begin{tabular}{ll|rrr|rrr|rr|rr|rr|rr}
				\multicolumn{2}{c|}{Problem} &	\multicolumn{3}{c|}{Info} & \multicolumn{3}{c|}{PSO} & \multicolumn{2}{c|}{SCP} & \multicolumn{2}{c|}{CCP} & \multicolumn{2}{c|}{S$l^1$QP} &  \multicolumn{1}{c}{GRB}\\
				Set & Spec   & States   & Trans.  & Par.      & tmin   & tmax   & \textbf{tavg}  & \textbf{t}    & iter & \textbf{t}     & iter & \textbf{t}     & iter  & \textbf{t}  \\\hline
				Brp           & $\p_{\leq 0.1}$       & 324    & 452  & \highlight{2}     &  0     & 0       &  \textbf{0}     & 0  &1  & 0   & 4 & 0   & 19 & 0 \\
				Brp           & $\p_{\leq 0.1}$       & 20999    & 29703  & \highlight{2}     &  2     & 4      &  \textbf{2}     & {6} & 1   & 180  & 87  & \TO   & --  & {5} \\
				Crowds           & $\p_{\leq 0.1}$       & 80    & 120  & \highlight{2}     &  1     & 1       &  1     & 2  &2   & \textbf{2}   &   3& 2   &   2 & 2 \\
				Nand       & $\p_{\leq 0.05}$       & 5447    & 7374  & \highlight{2}     &  0     & 0       &  \textbf{0}     & 1  &1   & 1   &   1 & 1   &   1 & 14 \\\hline
				Maze           & $\EV_{\leq 14}$       & 1303    & 2658  & \highlight{590} &     123 & 201 & 167  &  \textbf{1}  & 5 &   36  & 28  &   27   & 450 & \TO \\
				Maze           & $\EV_{\leq 10}$       & 1303    & 2658  & \highlight{590} &     -- &  -- & {\TO}             &  \textbf{2}  & 5 &   36  & 54  &   378   & 1047 & \TO \\
				Maze          & $\EV_{\leq 6}$       & 1303    & 2658  & \highlight{590}     &     -- &  -- & {\TO}          &  \textbf{3}  & 9  & 43     &  77&   \TO   & --  & \TO  \\
				Maze           & $\EV_{\leq 5.3}$       & 1303    & 2658  & \highlight{590}     &  -- &  -- & {\TO}        &  \textbf{7} & 25 &   152   & 192 &   \TO   & -- & \TO  \\
				Netw           & $\EV_{\leq 10}$       & 5040    & 13327  & \highlight{596}       &  213      & 273       &   243    &  3  & 1 & \textbf{3}     &   1 & \TO   & -- &\TO  \\
				Netw           & $\EV_{\leq 5}$       & 5040    & 13327  & \highlight{596}       &  --      & --       &   \TO    &  \textbf{3} & 2  & 6      & 3  & \TO   & -- &\TO \\
				Netw           & $\EV_{\leq 3.3}$       & 5040    & 13327  & \highlight{596}      &  --      & --   &   \TO     &  \textbf{20}   &27  & 193   & 97  &\TO   & -- & \TO \\
				%
				%
				Drone          & $\p_{\geq 0.90}$       & 4179    & 9414  & \highlight{1053}     & --      & --       &   \TO    &  \textbf{3}  & 3 & 23  & 13 & 96  & 207  & \TO \\				
				Drone           & $\p_{\geq 0.95}$       & 4179    & 9414  & \highlight{1053}     & --      & --       &   \TO    &  \textbf{3} & 4 &  177 & 99 & 995  & 2212  & \TO \\			
				Drone-fm         & $\p_{\geq 0.90}$    &     32403    & 70099  & \highlight{12286}     &  --      & --       &   \TO    &  \textbf{93}  & 3 & 1179   & 13& \TO   & -- & \TO  \\
				Drone-fm         & $\p_{\geq 0.95}$    &     32403    & 70099  & \highlight{12286}     &  --      & --       &   \TO    &  \textbf{96}  & 4 & \TO   & --& \TO   & --  & \TO  \\
				Drone-fm         & $\p_{\geq 0.99}$    &     32403    & 70099  & \highlight{12286}     &  --      & --       &   \TO    &  \textbf{160}  & 20 & \TO   & -- & \TO   & -- & \TO  \\
			\end{tabular}
	}}
\end{table*}


\subsection{Results for Further Models}
\paragraph{Benchmarks}
We include the standard pMC benchmarks from the \param website \footnote{\url{http://www.avacs.org/tools/param/}}, which contain two parameters.
We furthermore have a rich selection of strategy synthesis problems obtained from partially observable MDPs (POMDPs)~\cite{DBLP:journals/corr/abs-1710-10294}.
The first example is a maze setting (Maze), introduced in~\cite{DBLP:conf/icml/McCallum93}.
The objective is to reach a target location in the minimal expected time in a maze.
The initial state is randomized, and the robot can only get an observation if it hits a wall.
We consider an example with five memory nodes.

The second example from~\cite{DBLP:journals/corr/abs-1710-10294} schedules wireless traffic over a network channel (Netw), where at each time a scheduler generates a new packet for each user at each time~\cite{yang2011real}. 
The scheduler has partial observability of the quality of the channels.
The main trade-off is to probabilistically schedule the traffic to the user with the best or worst perceived channel condition.

Finally, in the drone example (Drone), the objective is to compute a policy to arrive at a target location while avoiding an intruder. 
The intruder is only visible within a radius, and the controller has to take account of the partial information.
We consider a memoryless and a finite-memory policy with five memory for observations (Drone-fm) to show the trade-offs between the computation time and the reachability probability.
The pMDP benchmarks originate from the \param website, or as parametric variants to existing \prism case studies.

\paragraph{Results}

Similar to Table~\ref{tab:results_satellite}, Table~\ref{tab:results_pmc} contains an overview of the results for pMCs.
Table~\ref{tab:results_pmdp} additionally contains the number of actions (Actions) for pMDPs.


For the Maze example, PRISM-POMDP found a solution that satisfies the $6$ threshold in $1.1$ seconds, performing better than all methods considered.
However, PRISM-POMDP ran out of memory in the other POMDPs.
We observe the effect of including memory clearly in the Drone benchmark.
In the examples with memoryless strategies, both SCP and CCP methods can find solutions for lower thresholds much faster than with finite-memory strategies.
With a threshold of $0.99$, PSO, CCP and S$l^1$QP ran out of time before finding a solution.
On the other hand, only SCP can find a solution for the $0.99$ threshold with a finite-memory strategy before the time-out.

Specifically, in most of the considered examples, SCP outperforms other methods, especially on the models with a large number of parameters.
The empirical results demonstrate that in the examples where S$l^1$QP successfully finds a solution, it requires significantly more steps than SCP.
The reason is, the linearization in S$l^1$QP is only accurate on a very small region around the previous solution. 
The algorithm prevents it from taking a large step to ensure that the linearization is accurate in all iterations.
On the other hand, by incorporating model checking results, SCP can take larger steps for faster convergence while ensuring the soundness of the solutions using the model checking results.
On a few instances, such as Crowds and Network, CCP can outperform the other methods.
Similarly, on instances with $2$ parameters, PSO is the best method in terms of the performance.
Overall, \emph{we observe that the SCP method is significantly faster than PSO, S$l^1$QP and CCP on benchmarks with many parameters}.

We show the convergence of the CCP and SCP in Figures~\ref{fig:conv_netw}--\ref{fig:conv_wlan} on three benchmarks with different thresholds.
We show the expected cost in each iteration for both methods until they find a feasible solution.
In all of these examples, SCP requires fewer iterations compared to CCP to compute a feasible parameter instantiation.
The obtained expected cost for SCP is also less than CCP after each iteration.
We observe that the objective does not decrease for SCP for several iterations in each figure, due to rejecting the parameter instantiation and contracting the trust region for these iterations.
However, SCP can minimize the expected cost after sufficiently reducing the size of the trust region to approximate the nonconvex problem accurately.

Finally, solving the nonconvex QCQP using Gurobi yields a similar runtime with SCP and CCP in pMCs with two parameters.
However, Gurobi failed to find a solution for all pMCs obtained from POMDPs before the time limit.

\begin{table*}[t]
	\centering
	\caption{pMDP benchmark results}
	\label{tab:results_pmdp}
	\footnotesize{
		\scalebox{1.00}{
			\begin{tabular}{ll|rrrr|rrr|rr|rr|rr|rr}
				\multicolumn{2}{c|}{Problem} &	\multicolumn{4}{c|}{Info} & \multicolumn{3}{c|}{PSO} & \multicolumn{2}{c|}{SCP} & \multicolumn{2}{c|}{CCP} & \multicolumn{2}{c|}{S$l^1$QP} &  \multicolumn{1}{c}{GRB}\\
				Set & Spec   & States   & Actions &  Trans.  & Par.      & tmin   & tmax   & \textbf{tavg}  & \textbf{t}    & iter & \textbf{t}     & iter & \textbf{t}     & iter  & \textbf{t}  \\\hline
				
				BRP  & $\p_{\leq 0.1}$ &18369 & 18510 & 24950 & \highlight{2}  & 1 & 3 & \textbf{1} & 2 & 3 & 180  & 87 &  2  & 1 & 2 \\  
				Coin  &$\p_{\geq 0.9}$ & 22656 & 60544 & 75232 & \highlight{2} & -- & -- & \TO & \textbf{5}& 1 & 36  & 3 & 11 & 2  & \TO \\ 
				Coin  &$\p_{\geq 0.99}$ & 22656 & 60544 & 75232 & \highlight{2} & -- & -- & \TO & \textbf{5}& 1 & 170  & 13& 25 & 6  & \TO \\ 
				CSMA  & $\EV_{\leq 68.9}$ & 7958 & 7988 & 10594 & \highlight{26} & n.s. & n.s. & n.s. & \textbf{3} & 44 & 19  & 28 & 19  & 498 & \TO \\
				CSMA  & $\EV_{\leq 68.7}$ & 7958 & 7988 & 10594 & \highlight{26} & n.s. & n.s. & n.s. & 86  & 916 & \textbf{19}  & 35  & 22  & 673 &\TO \\
				CSMA  & $\EV_{\leq 68.4}$ & 7958 & 7988 & 10594 & \highlight{26} & n.s. & n.s. & n.s. & \TO & -- & \textbf{37}  & 58 & \TO & -- & \TO \\
				Virus  & $\EV_{\leq 8.1}$ & 761 & 2606 & 5009 & \highlight{18} &  14 & 22 & 16 & \textbf{1} & 3 & 2 & 2 & 504 & 1117 & \TO \\
				WLAN  & $\EV_{\leq 2450}$ & 2711 & 3677  & 4877 & \highlight{15} & n.s. & n.s. & n.s. & \textbf{2} &  16 & 18 & 33 & \TO & --& \TO \\
				WLAN  & $\EV_{\leq 2404}$ & 2711 & 3677  & 4877 & \highlight{15} & n.s. & n.s. & n.s. & \textbf{18} &  146 & \TO & -- & \TO & -- & \TO \\
				\hline
			\end{tabular}
	}}
\end{table*}

\begin{figure}[t]
	\centering
	\begin{tikzpicture}
	\begin{axis}[
	xlabel={Number of Iterations},
	ylabel={Expected Cost},
	grid=major,
	ymode=log,
	minor tick num=2,
	xmin=-5,
	xmax=102,
	height=4.0cm,
	width=8.5cm,
	extra y ticks={3.3},
	extra y tick labels={$3.3$},
		yticklabels ={$3.3$,$6$,$10^2$,$10^3$,$10^4$},
	label style={font=\bf},
	legend style={nodes={scale=1.2}}]
	\addplot  [blue,line width=0.85mm, mark=*,mark size=0.4pt] table [x=x,y=y] {data/network_ccp.dat};
	\addlegendentry{CCP}
	\addplot  [red,line width=0.85mm, mark=triangle,mark size=0.4pt] table [x=x,y=y] {data/network_scp.dat};
	\addlegendentry{SCP}
	\addplot[mark=none,dashed,line width=1.25mm,  gray, domain=-5:105, samples=2] {3.3};
	\end{axis}
	\end{tikzpicture}
	\caption{The obtained expected cost of SCP and CCP versus the number of iterations for the Network example with a threshold of $3.3$. SCP can find a feasible solution with fewer iterations than CCP on this example.}
	\label{fig:conv_netw}
\end{figure}
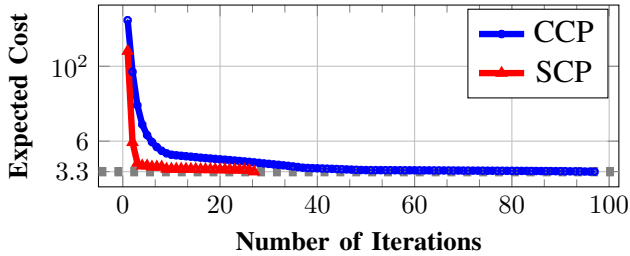

For the pMDP benchmarks, see Table~\ref{tab:results_pmdp}, we observe a similar pattern, where the SCP method is significantly faster than PSO and CCP in all but one benchmark.
For the CSMA benchmark, the SCP method fails to find a feasible solution for the  thresholds $68.7$ and $68.4$, whereas CCP can find a feasible solution in seconds.
In this benchmark, it takes $916$ iterations for SCP to find a solution that induces a cost less than $68.7$, and SCP converges to a locally optimal solution that induces a cost higher than $68.4$, therefore not satisfying the specification.
The benchmarks CSMA and WLAN are currently not supported by PSO due to the nonrectangular well-defined parameter space.
Similar to the pMC benchmarks, Gurobi was able to solve a pMDP with $2$ parameters at a similar time than CCP and SCP.
However, Gurobi failed to find a solution for other pMDPs within the time limit.

\begin{table}[t]
	\caption{The running time and number of iterations for the SCP method to find a feasible solution to the specification $\p_{\geq 0.95}$ for the Satellite-fm benchmark with different values of $\delta$ and $\gamma$.}
	\label{tab:hyperparams}
	\centering 
			\scalebox{1.0}{
	\begin{tabular}{|c|c|c|c|c|}
		\hline
		\diagbox{$\gamma$}{$\delta$}      & $0.5$ & $1.0$ & $1.5$ & $2.0$                  \\ \hline
		$1.25$ & $285,22$ & $280,21$ & $185,14$ & $176,12$ \\ \hline
		$1.5$  & $166,12$ & $109,7$  & $193,14$ & $239,18$ \\ \hline
		$1.75$ & $167,12$ & $152,11$ & $223,16$ & $239,18$ \\ \hline
		$2.0$  & $145,10$ & $121,11$ & $463,36$ & $361,27$ \\ \hline
	\end{tabular}}
\end{table}


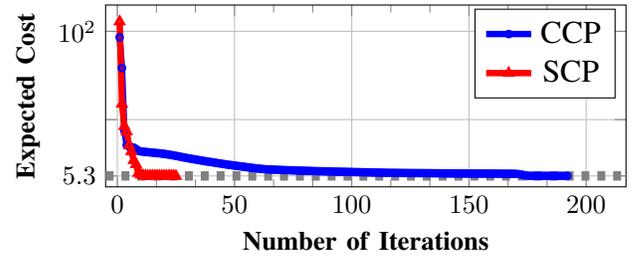
\begin{figure}[t]
	\centering
	\begin{tikzpicture}
	\begin{axis}[
	xlabel={Number of Iterations},
	ylabel={Expected Cost},
	grid=major,
	ymode=log,
	minor tick num=2,
	xmin=-5,
	xmax=217,
	ymin=3,
	height=4.0cm,
	width=8.5cm,
	extra y ticks={5.3},
	extra y tick labels={$5.3$},
	yticklabels ={$$,$$,$10^2$,$10^3$,$10^4$},
	label style={font=\bf},
	legend style={nodes={scale=1.2}}]
	\addplot  [blue,line width=0.85mm, mark=*,mark size=0.4pt] table [x=x,y=y] {data/maze_ccp.dat};
	\addlegendentry{CCP}
	\addplot  [red,line width=0.85mm, mark=triangle,mark size=0.4pt] table [x=x,y=y] {data/maze_scp.dat	};
	\addlegendentry{SCP}
	\addplot[mark=none,dashed,line width=1.25mm,  gray, domain=-5:217, samples=2] {5.3};
	\end{axis}
	\end{tikzpicture}
	\caption{The obtained expected cost of SCP and CCP versus the number of iterations for the Maze example with a threshold of $5.3$. SCP can find a feasible solution with fewer iterations than CCP on this example.}
	\label{fig:conv}
\end{figure}

\paragraph{Effect of values of hyperparameters $\delta$ and $\gamma$}
In Table~\ref{tab:hyperparams}, we report the running time and the number of iterations for the SCP method on finding a feasible solution with threshold $0.95$ for different values of the hyperparameters $\delta$ and $\gamma$.
In all of these examples, the running time and number of iterations for the SCP method vary for different values. 
For example, selecting $\delta=1.0$ and $\gamma=1.5$ gives the best performance, and the running time and the number of iterations are highest with $\delta=1.5$ and $\gamma=2.0$.
However, for all values of hyperparameters, SCP was able to compute a solution within a few minutes.

\paragraph{Effect of integrating model checking for CCP and SCP} 
Discarding the model checking results in our CCP implementation always yields time-outs, even for the relatively simple benchmark Maze with threshold $10$, which is solved with usage of model checking results within a minute.
Here, using model checking results thus yields a speed-up by a factor of at least 60.
For the Netw example, not using model checking increases the runtime by ten on average.
For the Drone examples that CCP solves before the timeout, we observe that CCP always exceeds time if we do not include the model checking results.

\emph{For SCP, we observed that model checking reduces the runtime of the procedure significantly in all examples and guarantees the solution's correctness.}
For the Drone examples, the size of the trust region becomes too small before reaching the required threshold without model checking, and SCP returns an infeasible solution for all thresholds.
On the other hand, we can observe that SCP can find a feasible solution within seconds or minutes if we include model checking.


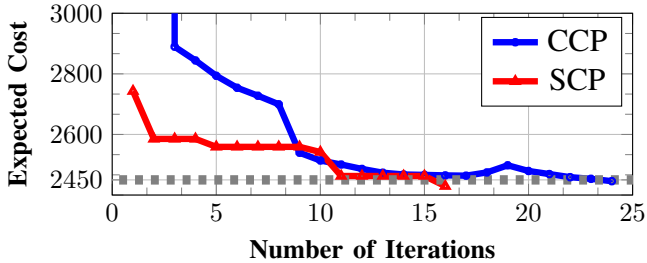
\begin{figure}[t]
	\centering
	\begin{tikzpicture}
	\begin{axis}[
	xlabel={Number of Iterations},
	ylabel={Expected Cost},
	grid=major,
	ymax=3000,
	ymin=2400,
	minor tick num=2,
	xmin=0,
	xmax=25,
	height=4.0cm,
	width=8.5cm,
	extra y ticks={2450},
	extra y tick labels={$2450$},
	yticklabels ={$$,$$,$2600$,$2800$,$3000$},
	label style={font=\bf},
	legend style={nodes={scale=1.2}}]
	\addplot  [blue, line width=0.85mm,mark=*,mark size=0.4pt] table [x=x,y=y] {data/wlan_ccp.dat};
	\addlegendentry{CCP}
	\addplot  [red, line width=0.85mm, mark=triangle,mark size=0.4pt] table [x=x,y=y] {data/wlan_scp.dat};
	\addlegendentry{SCP}
	\addplot +[mark=none,blue,line width=0.85mm] coordinates {(3, 2889.6303252624507) (3, 3000)};
	\addplot[mark=none,dashed,line width=1.25mm,  gray, domain=-5:217, samples=2] {2450};
	\end{axis}
	\end{tikzpicture}
	\caption{The obtained expected cost of SCP and CCP versus the number of iterations for the Wlan example with a threshold of $2450$. SCP can find a feasible solution with fewer iterations than CCP on this example.}
	\label{fig:conv_wlan}
\end{figure}
%
\subsection{Discussion}

The main results from the experiments is, \emph{a tuned variant of SCP improves the state-of-the-art.}
We observed that directly applying SCP methods does not yield a scalable method and may cause incorrect terminations.
To solve the nonconvex QCQP that originates from a pMDP efficiently, we needed to update parts of the model between each iteration instead of rebuilding. 
Additionally, Gurobi can use the results from the previous iterations, which improves the overall method's runtime.
Using state reduction techniques from model checking reduces the number of variables in each convex problem, which improves the runtime and the numerical stability of the procedure.
The critical ingredient is to incorporate model checking, which reduces the number of iterations significantly, as shown by the comparison with the S$l^1$QP algorithm.
It also allows a sound termination criterion, which cannot be realized by directly applying SCP methods as the linearized problem is not an over-approximation of the original QCQP.

These improvements to SCP yield a procedure that can handle much bigger problems than existing POMDP solvers. 
Our SCP method is also superior to sampling-based approaches in problems with many parameters.
It also significantly outperforms CCP~\cite{cubuktepe2018synthesis} in almost all examples considered.
Benchmarks with many parameters pose two challenges for sampling-based approaches: 
Sampling is necessarily sparse due to the high dimension, and optimal parameter valuations for one parameter often depend significantly on other parameter values.

%% file: conclusion.tex
\section{Conclusion and Future Work}\label{sec:conclusion}
We studied the applicability of convex optimization for parameter synthesis of parametric Markov decision processes (pMDPs). 
To solve the underlying nonconvex optimization problem efficiently, we proposed two methods that combine techniques from convex optimization and formal methods.
The experiments showed that our methods significantly improve the state-of-the-art and can solve large-scale satellite collision avoidance problems.
In the future, we will investigate how to handle the case if some of the parameters cannot be controlled and may be adversarial.
We will also apply our techniques to solving uncertain partially observable MDPs with different uncertainty structures.

%% file: appendix.tex
\appendix

\subsection{Proof of Theorem~\ref{thm:main}}

\begin{proof}
\label{sec:Proof_thm}

Since the assignment of the convexified DC problem in~\eqref{eq:min_qcqp_ccp} -- \eqref{eq:slackvariable_ccp} is feasible with 
	\begin{align*}
		\tau\sum_{\forall s\in S\setminus T}k_s=0,		
	\end{align*}
we know that the assignment is feasible for the QCQP in~\eqref{eq:min_mdp} -- \eqref{eq:probcomputation_mdp}. 
We will show that for every $s \in S$ we have $\pr(\mdp[\mathbf{v}],\finally T,s)\leq p_s$, for any feasible assignment for the QCQP in~\eqref{eq:min_mdp} -- \eqref{eq:probcomputation_mdp}.

For $s \in S$,  define $q_s = \pr_s(\mdp[\mathbf{v}],\finally T)$ (the probability to reach $T$ from $s$ in $\mdp[\mathbf{v}]$) and $x_s = q_s - p_s$. Let $S_{<}  = \{s \in S \mid p_s < q_s\}$.

For states $s \in T$ we have, by~\eqref{eq:targetprob_mdp} that $p_s=1=q_s = 1$, meaning that $s \not \in S_<$. For states $s$ with $q_s = 0$, \ie, states from which $T$ is almost surely not reachable, we have trivially $p _s \geq q_s$, also implying $s \not \in S_<$. Therefore,
for every $s \in S_<$, $p_s$ satisfies~\eqref{eq:probcomputation_mdp} and $T$ is reachable with positive probability.

Assume, for the sake of contradiction, that $S_{<} \neq \emptyset$, and let $x_{max}  =  \max\{x_s \mid s \in S\},$ and  $S_{max}  = \{s \in S \mid x_s  = x_{max}\}.$

The assumption that $S_< \neq \emptyset$ implies $x_{max} > 0$. Let $s \in S$ be such that $x_s = x_{max}$. Therefore, $s \in S_<$, and thus for all $\act \in \ActS(s)$, we have that 
 \begin{align}
 p_s \geq  \displaystyle \sum_{s'\in S}	\probmdp(s,\act,s')\cdot p_{s'}.
 \end{align}
On the other hand, there exists an $\act \in \ActS(s)$ such that 
 \begin{align}
\displaystyle q_s =\displaystyle \sum_{s'\in S} \probmdp(s,\act,s')\cdot q_{s'}.
 \end{align}
Thus, 
 \begin{align}
\displaystyle q_s-p_s & \leq \displaystyle \sum_{s'\in S} \probmdp(s,\act,s')\cdot (q_{s'}-p_{s'}), \label{eq:proofineq}
 \end{align}
which is equivalent to
 \begin{align}
\displaystyle x_s  \leq  \displaystyle \sum_{s'\in S} \probmdp(s,\act,s')\cdot x_{s'}.
 \end{align}
Since for all $\act \in \ActS(s)$, and $s' \in S$, we have that $\probmdp(s,\act,s') \geq 0$ and $\displaystyle \sum_{s'\in S} \probmdp(s,\act,s') = 1$, using~\eqref{eq:proofineq} we establish
 \begin{align}
x_{max} = x_s & \leq  \displaystyle \sum_{s'\in S} \probmdp(s,\act,s')\cdot x_{s'}\\
&\leq \displaystyle \sum_{s'\in S} \probmdp(s,\act,s')\cdot x_{max}\\
&\leq x_{max} \displaystyle \sum_{s'\in S} \probmdp(s,\act,s')= x_{max}.
 \end{align}
Which implies that all the inequalities are equalities, meaning 
 \begin{align}
x_{max} = x_s & =  \displaystyle \sum_{s'\in S} \probmdp(s,\act,s')\cdot x_{s'}\\
&= x_{max} \cdot \displaystyle \sum_{s'\in S} \probmdp(s,\act,s').\label{eq:proofeq}
 \end{align}
The equation in~\eqref{eq:proofeq} gives us that $x_{s'} = x_{max} > 0$ for every successor $s'$ of $s$ in $\mdp[\mathbf{v}]$. Since $s \in S_{max}$ was chosen arbitrarily, for every state $s \in S_{max}$, all successors of $s$ are also in $S_{max}$. As we established that $S_< \cap T = \emptyset$, it is necessary that $T$ is not reachable with positive probability from any $s \in S_{max}$, which is a contradiction with the fact that from every state in $S_<$, the set $T$ is reachable.
\end{proof}%
\vspace{-0.1cm}

\subsection{Proof of Theorem~\ref{thm:kkt}}
\begin{proof}
		We show that the NLP in \eqref{eq:min_mdp} -- \eqref{eq:probcomputation_mdp} satisfies linear independence constraint qualification, which states that the gradients of the active inequality and equality constraints are linearly independent at a locally optimal solution.
		\added{Without loss of generality, we assume that the pMDP is \emph{simple}, i.e., $\probmdp(s,\act,s')\in \lbrace{v,1-v | v \in V} \cup R_{[0, 1]}\rbrace$.
			
			First, the constraints \eqref{eq:well-defined_probs_mdp}--\eqref{eq:well-defined_probs_mdp1} reduce to the constraints $$\epsgraph \leq v \leq 1-\epsgraph \quad \forall v \in V,$$ and the gradients of the active constraints would therefore be independent at any solution $\forall v^*\in V$.   
			
			Second, we now focus on the constraints~\eqref{eq:probcomputation_mdp}.
			This constraint is affine in the probability variables $p_s$ for $s \in S\setminus T$ for a fixed value of the parameters in $V$.
			Further, it is known that for any value of the parameters, the solution for the probability variables $p_s$ for $s \in S\setminus T$ instantiated MDP is unique~\cite[Theorem 10.19, Theorem 10.100]{BK08}, after a preprocessing step removing all states $s \in S$ that reach a target state $t \in T$ with probability $0$ and $1$ using Algorithm 46 in~\cite{BK08}. Note that this preprocessing step also removes the constraint~\eqref{eq:targetprob_mdp}.
			
			Furthermore, the solution $\mathbf{p}^*$ for the probability variables can be obtained by solving the equation system $\mathbf{p}=\mathbf{P}\mathbf{p}+\mathbf{b},$}where $\mathbf{p}=(p_s)_{\forall s \in S\setminus T}$, i.e., the probability variables, the matrix $\mathbf{P}$ contains the transition probabilities for the states in $S\setminus T$ for the optimal deterministic strategy $\sched(s)$, i.e., $\mathbf{P}=(\probmdp(s,\sched(s),s'))_{\forall s, s' \in S\setminus T}$, and the vector $\mathbf{b}=(b_s)_{\forall s \in S\setminus T}$ contains the probabilities of reaching a state in the target set $T$ from any state $s \in S\setminus T$ within one step under the optimal action, i.e., $b_s=\sum_{t \in T}\probmdp(s,\sched(s),t)$.
		This equation system correspond to the active inequality constraints in~\eqref{eq:probcomputation_mdp}, which cam be written as 
		\begin{align*}
		p^*_s = \sum_{s'\in S}	\hat{\probmdp}(s,\sched(s),s')\cdot p^*_{s'}\quad	\forall s\in S\setminus T,
		\end{align*}
		where $\hat{\probmdp}$ denote the transition function for the instantiated pMDP, and $p^*_s$ is the unique solution for the probability variable at state $s$.
		Note that the solution for the probability variables is unique, and the above equation system is affine in the probability variables. 
		Therefore, the gradients of the above active inequality constraints are linearly independent at their unique solution $p_s^*$ for the probability variables $p_s$.
		
		Combining the two facts, we conclude that the NLP in \eqref{eq:min_mdp} -- \eqref{eq:probcomputation_mdp} satisfies the linear independence constraint qualification for any locally optimal solution.
\end{proof}

\subsection{Convergence of SCP Methods for Parameter Synthesis}

In this section, we show the convergence of a SCP method that is a variant of Fig.~\ref{fig:scpwithmc}. 
Similar to Section~VI-c, we denote our variable pair $(\mathbf{v}, \mathbf{p})$ by $x$ in this section.
Let the objective in~\eqref{eq:min_mdp} be $h_0(x)$, the inequality constraints in~\eqref{eq:probcomputation_mdp} be
\begin{align*}
h_i(x)\leq 0, \quad i=1,\ldots,m
\end{align*}
where $m$ is the number of inequality constraints.
As the other constraints in the QCQP~\eqref{eq:min_mdp}--\eqref{eq:probcomputation_mdp} are convex, we compactly represent the constraints in~\eqref{eq:targetprob_mdp}--\eqref{eq:probthreshold_mdp} as
\begin{align*}
x \in Q.
\end{align*}
Note that the convex set $Q$ is compact, which is a required assumption for the convergence analysis of SCP methods.

{\color{black}
\begin{algorithm}[t]
	\begin{algorithmic}[1]
		\Statex Initialize: Select $x_0 \in Q, \beta_0, \delta >0.$
		\State \textbf{Step 1} At each iteration $l$, solve the convex problem~\eqref{eq:appendix_obj}--\eqref{eq:appendix_cons} that is convexified around $x_l$ to compute a unique solution $x_{l+1}$.
		\State \textbf{Step 2}
		\If {$ h_i(x_l) +\nabla h_i (x_l)^{\top}(x-x_l)\leq 0$ for $i=1,\ldots,m$} 
		\State Accept this step, and set $\beta_{l+1}=\beta_l$.
		\Else
		\State Reject this iteration, set $x_{l+1}\leftarrow x_l$, and $\beta_{l+1} \leftarrow \beta_l + \delta.$
		\EndIf
		\caption{Sequential convex programming with regularization} 
		\label{alg:scp_app}
	\end{algorithmic}
\end{algorithm}}

\begin{remark}
We remark that we do not include the threshold constraint in~\eqref{eq:probthreshold_mdp}, as SCP methods require an initial feasible point for convergence analysis.
Such an initial feasible point can be obtained by a model checking step with any well-defined parameter instantiations without the threshold constraints.
\end{remark} 

We also note that the functions $h_i(x)$ are $C^2$ with Lipschitz continuous gradients as they are quadratic. 
For each $i=0,1,\ldots,m$, we denote by $L_i>0$ the Lipschitz constants of $\nabla h_i$.

At each iteration $l$, we solve the following convex problem with variables $x$ and $k$:
\begin{align}
\text{minimize} & \;\; h_0(x)+\beta_l k + \dfrac{\mu + \beta_l\mu'}{2}\Vert x_l - x\Vert^2_2\label{eq:appendix_obj}\\
\text{subject to} & \;\; h_i(x_l) +\nabla h_i (x_l)^{\top}(x-x_l)\leq k,\; i=1, \ldots, m,\\
& \;\; k \geq 0,\\
&\;\; x \in Q.\label{eq:appendix_cons}
\end{align} 
where $\mu > L_0$ and $\mu' \geq \max_{i=1,\ldots,m} L_i, \beta_l>0.$ 
We show the SCP method in Algorithm~\ref{alg:scp_app}.

We now state the convergence result of Algorithm~\ref{alg:scp_app}.

\begin{theorem}[]\cite[Theorem 2.2, Theorem 2.3]{bolte2016majorization}.
	The sequence $\lbrace x_l \rbrace$ generated by Algorithm~\ref{alg:scp_app}. converges to a feasible point $x_{\infty}$ that satisfies the KKT conditions in Theorem~\ref{thm:kkt} for the QCQP in~\eqref{eq:min_mdp}--\eqref{eq:well-defined_probs_mdp1} and \eqref{eq:probcomputation_mdp}.
	Furthermore, the convergence rate is in the form
	\begin{align*}
	&\Vert x_k - x_{\infty} \Vert_2 = O(q^k), \text{ with } q \in (0 ,1),\\
	&	\Vert x_k - x_{\infty} \Vert_2 = O(1/k^\xi), \text{ with } \xi >0.
	\end{align*}
\end{theorem}
